\documentclass[12pt,reqno]{amsart}
\usepackage{amsmath,amssymb}
\usepackage{xcolor}
\usepackage{graphicx}
\usepackage{latexsym,amssymb}
\usepackage{mathrsfs}
\usepackage[abbrev]{amsrefs}

\newtheorem{thm}{Theorem}[section]

\newtheorem{prop}[thm]{Proposition}
\newtheorem{lem}[thm]{Lemma}

 \newenvironment{pf}
    {{\noindent \bf Proof. }}{\hfill $\Box$}
\numberwithin{equation}{section}
\numberwithin{thm}{section}
\begin{document}

\begin{center}\large \bf 
Unique existence of solutions to the inviscid SQG equation in a critical space
%Local solutions and small global solutions 
%\\ 
%for critical SQG on a unit ball
\end{center}

\footnote[0]
{


E-mail: t-iwabuchi@tohoku.ac.jp

}
\vskip5mm

\begin{center}
%{\sf 
Tsukasa Iwabuchi 
%}

\vskip2mm

Mathematical Institute, 
Tohoku University\\
Sendai 980-8578 Japan

\end{center}

\vskip5mm

\begin{center}
\begin{minipage}{135mm}
\footnotesize
{\sc Abstract. } 
We study the Cauchy problem for the surface quasi-geostrophic (SQG) equations in a two-dimensional bounded domain with the homogeneous Dirichlet boundary condition. We establish the unique existence of strong solutions in the critical Besov space $\dot B^2_{2,1}$, which is embedded in $C^1$. The proof is based on spectral localization using dyadic decomposition associated with the Dirichlet Laplacian. We obtain the solution by establishing uniform estimates for a sequence of solutions to the equation with a regularized nonlinear term.

\end{minipage}
\end{center}

%%%%%%%%%%%%%%%%%%%%%%%%%%%%%%%%%%%%%
%%%%%%%%%%%%%%%%%%%%%%%%%%%%%%%%%%%%%
\section{Introduction}

Let $\Omega \subset \mathbb R^2$ be a bounded domain with a smooth boundary. 
We consider the surface quasi-geostrophic equation: 
\begin{equation}\label{QG1}
\displaystyle 
 \partial_t \theta 
  + (u \cdot \nabla ) \theta  =0, 
  \quad u= \nabla ^{\perp} \Lambda _D ^{-1} \theta, 
 \qquad  t > 0 , x \in \Omega, 
\end{equation}
\begin{equation}\label{QG2}
 \theta(0,x) = \theta_0(x) , 
\qquad  x \in \Omega, 
\end{equation}
where 
$\nabla^\perp := (-\partial_{x_2} , \partial_{x_1})$, 
and $\Lambda_D$ denotes the square root of the Dirichlet Laplacian. 
 The equations are recognized as a fundamental model in geophysical fluid dynamics (see~\cite{La_1959,Pe_1979}). 
 The purpose of this paper is to establish the existence and uniqueness of strong solutions in the critical Besov space $\dot B^2_{2,1}(A_D)$ associated with the Dirichlet Laplacian $A_D$.

Let us recall the existing literature for the cases $\Omega = \mathbb R^2$ and $\Omega = \mathbb T^2$. 
The local-in-time existence of strong solutions in the Sobolev space $H^s_p(\mathbb R^2)$ holds for $s > 1 + 2/p$. This result is obtained by an approach similar to that used for the Euler equations, as demonstrated by Kato and Ponce \cite{KaPo-1988}. 
The exponent $s = 1 + 2/p$ corresponds to the critical regularity required to obtain a closed estimate for the nonlinearity; 
this regularity is considered optimal in the sense that the Sobolev space $H^s_p(\mathbb R^2)$ is embedded in $C^1(\mathbb R^2)$ only when $s > 1 + 2/p$. 
In the critical case $s = 1+ 2/p$, we can construct the strong solutions in the Besov spaces $\dot B^{1+d/p}_{p,1}(\mathbb R^2)$, which is obtained by an approach for the Euler equations by Vishik~\cite{Vis-1998,Vis-1999} Chae~\cite{Chae-2004}, Pak and Park~\cite{PaPa-2004}.

Regarding weak solutions, the existence of global weak solutions in $L^2$ was established by Resnick \cite{Resn_1995}, and the result was extended to the $L^p$ framework by Marchand \cite{Mar-2008}. 
Subsequently, the non-uniqueness of weak solutions was proved by Buckmaster, Shkoller, and Vicol \cite{BuShVi-2019} (see also Isett and Ma \cite{IsMa-2021}). 
The non-uniqueness of stationary weak solutions was further investigated by Cheng, Kwon, and Li \cite{ChKwLi-2021}.

In the critical case $s = 1 + 2/p$, the strong ill-posedness in Sobolev spaces for the Euler equations was established by Bourgain and Li \cite{BourLi-2015} (see also~\cite{BourLi-2021}). 
 They demonstrated the discontinuity of the solution map with respect to the initial data. 
 We also refer to the work of Elgindi and Jeong \cite{ElJ2-2017} for related results.
In the $W^{1,\infty}(\mathbb R^2)$ framework, Elgindi and Masmoudi \cite{ElMa-2020} proved the ill-posedness for the surafce quasi-geostrophic equations. 
More recently, Jeong and Kim \cite{JeKi-2024} demonstrated norm inflation in the critical space $H^2(\mathbb T^2) \cap W^{1,\infty}(\mathbb T^2)$ and showed the non-existence of a solution operator from $H^2(\mathbb T^2)$ to $L^\infty(0,T; H^2(\mathbb T^2))$. 
Furthermore, C\'ordoba, Mart\'inez-Zoroa, and O\.za\'nski \cite{CoMaOz-2025} constructed a smooth solution for initial data $\theta_0 \in H^s(\mathbb R^2)$ with $3/2 < s < 2$ such that the solution $\theta(t)$ belongs to $H^{\bar s}(\mathbb R^2)$ for $\bar s = s / (1+ct)$ for some $c > 0$, but fails to remain in $H^{\bar s + \varepsilon }(\mathbb R^2)$ for $\varepsilon > 0$.

We also recall several known results for domains with a boundary $\partial \Omega$. 
For a bounded domain $\Omega$ with a smooth boundary, the existence of global weak solutions in $L^2$ was established by Constantin and Ignatova \cite{CoNg-2018}. 
Constantin, Ignatova, and Nguyen \cite{CoIgNg-2018} subsequently studied the inviscid limit of the equations with fractional dissipation. 
The existence of strong solutions in $H^1_0(\Omega) \cap W^{2,p}(\Omega)$ for $1 < p < \infty$ was obtained in \cite{CoNg-2018-2}. 
Furthermore, Jeong, Kim, and Yao \cite{JeKiYa-2025} investigated the problem in the half-plane and established a non-existence theorem for $C^1$ solutions given initial data in $C^\infty_c(\overline{\mathbb R^2_+})$. 
A crucial feature of their result is that smooth solutions cannot exist if the initial data does not vanish on the boundary.

In this paper, we establish the existence and uniqueness of solutions in the space $\dot B^2_{2,1}(A_D)$, which is a subspace of $H^2(\Omega)$. When $\Omega$ is a bounded domain, results in such a critical space have not been previously established. While well-posedness in $\mathbb{R}^2$ is a classical result, its extension to bounded domains is far from a straightforward generalization. The breakdown of the Fourier-based Littlewood-Paley decomposition necessitates a delicate analysis based on the spectral theory for the Dirichlet Laplacian. This approach must account for not only the regularity but also the compatibility of the boundary conditions with the nonlinearity. The space $\dot B^2_{2,1}(A_D)$ is defined in terms of the Dirichlet Laplacian, and the solutions are constructed in the class $C([0,T]; \dot B^2_{2,1}(A_D))$. A precise definition of the Besov spaces $\dot B^s_{p,q}(A_D)$ associated with $A_D$, along with their fundamental properties, will be provided in Section 2.

 To define these spaces, let $\phi_0 \in C_c^\infty(\mathbb R)$ be a function such that 
 \[
 {\rm supp \, } \phi_0 \subset [2^{-1}, 2] \quad \text{and} \quad \sum_{j \in \mathbb Z} \phi_j(\lambda) = 1 \quad \text{for all } \lambda > 0,
 \]
 where we define 
 \begin{equation}\label{0304-2}
 	\phi_j(\lambda) := \phi_0 \left( \frac{\lambda}{2^j} \right) \quad \text{for } j \in \mathbb Z, \, \lambda \in \mathbb R.
 \end{equation}
 Let $A_D = -\Delta$ denote the Dirichlet Laplacian and let $\Lambda_D := \sqrt{A_D}$. 
 For $s \in \mathbb R$ and $1 \le p, q \le \infty$, the norm of the homogeneous Besov space $\dot B^s_{p,q}(A_D)$ is defined by 
 \[
 \| f \|_{\dot B^s_{p,q}(A_D)} := \left\{ \sum_{j \in \mathbb Z} \left( 2^{sj} \| \phi_j(\Lambda_D) f \|_{L^p} \right)^q \right\}^{1/q}.
 \]

\begin{thm}
\label{thm:1} 
Let $\theta_0 \in \dot B^2_{2,1}(A_D)$. Then, there exists a time $T > 0$ and a solution $\theta \in C([0,T], \dot B^2_{2,1}(A_D)) \cap C^1 ([0,T], \dot B^1_{2,1}(A_D))$ to the Cauchy problem \eqref{QG1}--\eqref{QG2}. 
Furthermore, the solution is unique in the class 
$C([0,T], W^{1,\infty}(\Omega)) \cap C^1 ([0,T], L^2)$. 

\end{thm}

\noindent 
{\bf Remark. }  (1) 
Due to the properties of the Dirichlet Laplacian, the boundary condition $\theta|_{\partial \Omega} \equiv 0$ is preserved as long as the solution exists. 
This is the first result to establish the unique existence of solutions in a critical space embedded in $C^1(\Omega)$, where $\Omega$ is bounded. 
Furthermore, the non-existence result in the half-plane \cite{JeKiYa-2025} for $\theta_0|_{\partial \mathbb R^2_+} \not\equiv 0$ suggests that maintaining the Dirichlet boundary condition for all $t > 0$ is essential for the existence of smooth solutions.

 \vskip2mm 
 
\noindent 
(2) In Theorem~\ref{thm:1}, $W^{1,\infty}(\Omega)$ denotes the standard Sobolev space consisting of bounded functions whose first-order derivatives are also bounded. 
It is worth noting that the uniqueness result in Theorem~\ref{thm:1} holds for functions regardless of their boundary behavior. 
Furthermore, the solution space $C([0,T]; \dot B^2_{2,1}(A_D))$ is embedded in the class $C([0,T]; W^{1,\infty}(\Omega)) \cap C^1([0,T]; L^2(\Omega))$, where uniqueness is established.

\vskip2mm 

\noindent 
(3) 
The following approximation property holds for solutions with higher regularity. For $u \in \dot{B}^{2}_{2,1}(A)$, there exists a constant $C > 0$ such that
\[
\left\| u - \sum_{j \leq J} \phi_j (\sqrt{A}) u \right\|_{L^2} 
\leq C 2^{-2 J}
\]
for all $J \in \mathbb{N}$. 
Furthermore, since the spectrum of the Dirichlet Laplacian on a bounded domain consists of eigenvalues, the partial sum above is a finite linear combination of the corresponding eigenfunctions.

\vskip3mm 

Let us give several comments on the proof of Theorem~\ref{thm:1}. Instead of the original nonlinear term $(u \cdot \nabla)\theta$ with $u = \nabla^\perp \Lambda_D^{-1} \theta$, we consider a regularized term:
\[
N_\mu (\theta, \theta ) := (1+ \mu A_D)^{-1} 
\Big( \left( \nabla^\perp \Lambda_D^{-1}\theta \cdot \nabla \right) (1+\mu A_D)^{-1} \theta \Big), \quad \mu > 0.
\]
Using this, we construct a sequence of solutions in $C([0,T]; \dot B^2_{2,1}(A_D))$ (see Proposition~\ref{prop:1010-5}). 
In addition, we establish uniform boundedness with respect to the parameter $\mu > 0$ in a stronger topology of $L^\infty(0,T; \dot B^2_{2,1}(A_D))$. 
By taking a subsequence that converges in the weak-$*$ sense, we obtain a local-in-time solution.
To this end, Lemma~\ref{lem:1126-5} provides nonlinear estimates using spectral localization via dyadic decomposition, which is essential to avoid derivative loss. 
In this domain setting, we utilize a spectral restriction based on the resolvent:
\begin{equation}\label{0126-1}
	\psi _j (\Lambda _D) 
	:= (1+ 2^{-2j-2} A_D)^{-1} - (1+2^{-2j}A_D)^{-1},
\end{equation}
instead of the standard smooth functions with compact support. This resolvent-based approach is a powerful tool for handling such problems in bounded domains. 
We refer to \cite{Iw-2022}, which introduced this method for commutator estimates. In the present paper, we use $\psi_j(\Lambda_D)^{1/2}$ to characterize an equivalent norm for the Besov space $\dot B^2_{2,1}(A_D)$ (see Lemma~\ref{lem:0128-3}).

\bigskip 

This paper is organized as follows. 
In section 2, we recall the definition of Besov spaces associated with the Dirichlet 
Laplacian, and several properties for the boundary value of functions, and bilinear estimates  with derivatives. 
In section 3, we prove Theorem~\ref{thm:1}.

\bigskip

\noindent 
{\bf Notations.} 
For $k\in \mathbb N$, $W^{k,p}(\Omega)$ is the standard Sobolev space without any boundary condition and consists of all $f$ such that $f $ and its all the derivatives belong to $ L^p$. 
$H^k(\Omega)$ is $W^{k,p}(\Omega)$ for $p=2$, and 
$H^1_0 (\Omega)$ is the completion of $C_0^\infty (\Omega)$ in $H^1 (\Omega)$. 
$\dot B^s_{p,q}(A_D)$ is the Besov space associated with the Dirichlet Laplacian. 
$\{\phi_j \}_{j \in\mathbb Z}$ is defined by \eqref{0304-2} and 
$
S_j := \sum _{k \leq j} \phi_j(\Lambda _D)
$. 
$\psi_j(\Lambda _D)$ is the spectral restriction operator defined with the resolvent by $\psi _j (\Lambda _D) 
=: (1+ 2^{-2j-2} A_D)^{-1} - (1+2^{-2j}A_D)^{-1}
$.

\section{Preliminaries}

This section provides several propositions and lemmas that are essential for the proof of Theorem~\ref{thm:1}. 
In Subsection 2.1, we recall the definition of Besov spaces associated with the Dirichlet Laplacian and summarize their fundamental properties. 
Subsection 2.2 is devoted to the study of Besov spaces with low regularity indices, specifically the case where smooth functions belong to these spaces regardless of their boundary values. 
In Subsection 2.3, we establish product estimates in Besov spaces for functions and their derivatives. These estimates are particularly tailored for the low regularity case, where the functions do not necessarily satisfy the Dirichlet boundary condition.

\subsection{Besov spaces}\label{subsec:besov}

We recall the definition and basic properties of Besov spaces associated with the Dirichlet Laplacian (see~\cite{IMT-2019}). Let $A_D = -\Delta_D$ be the Dirichlet Laplacian on $L^2$ defined by
\[
\begin{cases}
	D(A_D) := \{ f \in H^1_0(\Omega) \mid \Delta f \in L^2 \}, \\
	A_D f := -\Delta f, \quad f \in D(A_D).
\end{cases}
\]
We define the operator 
\[
\Lambda_D := \sqrt{A_D}. 
\]
Since we consider a bounded domain $\Omega$ with the Dirichlet boundary condition, the infimum of the spectrum of $A_D$ is strictly positive. This positivity implies that the homogeneous and non-homogeneous Besov spaces are equivalent as sets and possess equivalent norms. For simplicity of notation, we adopt the homogeneous framework throughout this paper.
Let $\{ \phi_j \}_{j \in \mathbb Z}$ be defined by \eqref{0304-2}. 

\vskip3mm 

\noindent {\bf Definition.} 
(i) Let $\mathcal{Z}_D$ be the space of test functions defined by
\[
\mathcal{Z}_D := \{ f \in L^1 \cap D(A_D) \mid q_m(f) < \infty \text{ for all } m \in \mathbb{N} \},
\]
where the semi-norms $q_m$ are given by
\[
q_m(f) := \sup_{j \in \mathbb{Z}} 2^{m|j|} \| \phi_j(\Lambda_D) f \|_{L^1(\Omega)}.
\]
(ii) Let $\mathcal{Z}_D'$ be the topological dual of $\mathcal{Z}_D$.

\bigskip 

Although the operator $A_D$ is initially defined on $L^2$, the boundedness properties on $L^p$ are well-established. These properties allow us to justify the definition of $\mathcal{Z}_D$ as a Fr{\'e}chet space and to employ the functional calculus on $\mathcal{Z}_D$ and $\mathcal{Z}_D'$ (see \cite{IMT-2019}). The following lemma provides the boundedness of the spectral restriction operator $\phi_j(\Lambda_D)$ on Lebesgue spaces.

\begin{lem}{\rm(}\cite{FuIw-preprint,IMT-2018}{\rm)} \label{lem:0128-1}
	Let $1 \leq r \leq p \leq \infty$. There exists a positive constant $C $ such that
	\[ 
	\| \phi_j(\Lambda_D) \|_{L^r \to L^p} \leq C 2^{2\left(\frac{1}{r}-\frac{1}{p}\right)j} 
	\]
	for all $j \in \mathbb{Z}$. Furthermore, for $\alpha = 1, 2$, it holds that
	\[ 
	\| \nabla^\alpha \phi_j(\Lambda_D) \|_{L^r \to L^p} \leq C 2^{\alpha j + 2\left(\frac{1}{r}-\frac{1}{p}\right)j} 
	\]
	for all $j \in \mathbb{Z}$.
\end{lem}

We then define Besov spaces associated with the Dirichlet Laplacian on $\Omega$. 

\vskip2mm

\noindent 
{\bf Definition}. (Besov spaces)
Let $s \in \mathbb R$ and $1 \leq p,q \leq \infty$. 
$\dot B^s_{p,q} = \dot B^s_{p,q}(A_D)$ is defined by 
\[
\dot B^s_{p,q}(A_D) :=  
\{ f \in \mathcal Z_D' \, | \, 
\| f \|_{\dot B^s_{p,q}(A_D) } < \infty\}, 
\]
where 
\[
\| f \|_{\dot B^s_{p,q}(A_D)} 
:= 
\Big\|  
\Big\{ 2^{sj} \| \phi_j(\Lambda _D) f \|_{L^p(\mathbb R^2_+)} 
\Big\}_{j \in \mathbb Z}
\Big\| _{\ell^q(\mathbb Z)}.
\]

\bigskip 

It was established in \cite{IMT-2019} that $\dot{B}^s_{p,q}(\Lambda_D)$ is a Banach space and satisfies standard properties, such as lift properties and Sobolev-type embedding theorems, similar to those in the whole space case. We list below several properties that are required for our proof.

\begin{prop}\label{prop:0128-2} Let $s \in \mathbb R$ and $1\leq p ,q \leq \infty $. 
\begin{enumerate}
	\item $\dot B^s_{p,q}(A_D)$ is a Banach space. 
	\item Let $1 \leq p,q < \infty$, $1 < p',q' \leq \infty$ and $1/  p+ 1/p' = 1/q+q/q' = 1$.  
		The dual space of $\dot B^s_{p,q}(A_D) $ is $ \dot B^{-s}_{p',q'}(A_D)$. 
	\item Let $r \leq p$. Then 
	$\dot B^{s+2(\frac{1}{r}-\frac{1}{p})}_{r,q}(A_D) 
	\hookrightarrow \dot B^s_{p,q}(A_D)$. 
	\item 
	Let $s_0 > 0$ and $1 \leq q_0 \leq \infty$. Then 
	$\dot B^{s+s_0}_{p,q_0}(A_D) \hookrightarrow \dot B^s_{p,q}(A_D)$. 
\end{enumerate}
\end{prop}

\noindent 
{\bf Remark. } 
The embedding (4) in Proposition~\ref{prop:0128-2} follows from the corresponding embedding for non-homogeneous Besov spaces, combined with the equivalence between the homogeneous and non-homogeneous types. This equivalence is ensured by the positivity of the infimum of the spectrum of $A_D$.

\begin{lem}\label{lem:0128-3}
\begin{enumerate}	
	\item Let	$\alpha = 0, 1,2$ and $1 \leq r \leq p \leq \infty$. A positive constant $C$ exists such that 
\[
\| \nabla ^\alpha f \|_{L^p} \leq C \| f \|_{\dot B^{\alpha +2(\frac{1}{r}-\frac{1}{p})}_{r,1}(A_D)}
\quad \text{for all } f \in \dot B^{2(\frac{1}{r}-\frac{1}{p})}_{r,1}(A_D). 
\]

	\item  Let $\psi_j(\Lambda_D)$ be defined by \eqref{0126-1}. 
	Suppose that  $|s-2| < 1$ and $1 \leq p,q \leq \infty$. 
	For every $f \in \dot B^s_{p,q}(A_D)$, we have the following norm equivalence. 
	\[
	\| f \|_{\dot B^s_{p,q}(A_D)} \simeq 
	\Big\{  \sum _{j \in \mathbb Z} \Big(  2^{(s-2)j}\| \psi_j(\Lambda _D)^{\frac{1}{2}} \Delta f \|_{L^p} \Big)
	\Big\}^{\frac{1}{q}}. 
	\]
\end{enumerate}  
\end{lem}

\begin{pf}
	We define $\Phi_j := \phi_{j-1} + \phi_j + \phi_{j+1}$ and note that $\phi_j(\Lambda_D) = \Phi_j(\Lambda_D) \phi_j(\Lambda_D)$. 
	
	\noindent 
	(1) Let $f \in \dot{B}^{\alpha + 2(\frac{1}{r}-\frac{1}{p})}_{r,1}(A_D)$. We utilize the resolution of identity in $\mathcal{Z}'_D$ (see Lemma~4.5 in \cite{IMT-2019}):
	\[
	f = \sum_{j \in \mathbb{Z}} \phi_j(\Lambda_D) f.
	\]
	By applying Lemma~\ref{lem:0128-1} to the term $\nabla^\alpha \Phi_j(\Lambda_D)$, we obtain:
	\[
	\begin{split}
		\| \nabla^\alpha f \|_{L^p} 
		&\leq \sum_{j \in \mathbb{Z}} \| \nabla^\alpha \Phi_j(\Lambda_D) \phi_j(\Lambda_D) f \|_{L^p} \\
		&\leq C \sum_{j \in \mathbb{Z}} 2^{\alpha j + 2(\frac{1}{r}-\frac{1}{p})j} \| \phi_j(\Lambda_D) f \|_{L^r} \\
		&= C \| f \|_{\dot{B}^{\alpha + 2(\frac{1}{r}-\frac{1}{p})}_{r,1}(A_D)}.
	\end{split}
	\]

\noindent 
(2) Let $f \in \dot{B}^s_{p,q}(A_D)$. By the lift property for $\dot{B}^s_{p,q}(A_D)$, we have
\[
\| f \|_{\dot{B}^s_{p,q}(A_D)} \simeq \| A_D f \|_{\dot{B}^{s-2}_{p,q}(A_D)} = \| \Delta f \|_{\dot{B}^{s-2}_{p,q}(A_D)}. 
\]
We note that $\phi_j(\Lambda_D) \psi_j(\Lambda_D)^{-\frac{1}{2}}$ is an operator defined via dyadic scaling, associated with a function in $C_0^\infty(\mathbb R)$. Applying Lemma~\ref{lem:0128-1} to $\phi_j(\Lambda_D) \psi_j(\Lambda_D)^{-\frac{1}{2}}$, we obtain
\[
\begin{split}
	\| \phi_j(\Lambda_D) \Delta f \|_{L^p} 
	&= \| \phi_j(\Lambda_D) \psi_j(\Lambda_D)^{-\frac{1}{2}} \psi_j(\Lambda_D)^{\frac{1}{2}} \Delta f \|_{L^p} 
	\leq C \| \psi_j(\Lambda_D)^{\frac{1}{2}} \Delta f \|_{L^p},
\end{split}
\]
which implies
\[
\| f \|_{\dot{B}^s_{p,q}(A_D)} \leq C \left\{ \sum_{j \in \mathbb{Z}} \left( 2^{(s-2)j} \| \psi_j(\Lambda_D)^{\frac{1}{2}} \Delta f \|_{L^p} \right)^q \right\}^{\frac{1}{q}}.
\]

For the converse inequality, we write 
\[
\psi_j(\Lambda_D)^{\frac{1}{2}} \Delta f = \sum_{k \in \mathbb{Z}} A_D \psi_j(\Lambda_D)^{\frac{1}{2}} \Phi_k(\Lambda_D) \Big( \phi_k(\Lambda_D) f \Big). 
\]
Based on the following explicit relation:
\[
A_D \psi_j(\Lambda_D)^{\frac{1}{2}} = A_D \left\{ \frac{3}{4} \cdot 2^{-2j} A_D (1 + 2^{-2j-2} A_D)^{-1} (1 + 2^{-2j} A_D)^{-1} \right\}^{\frac{1}{2}},
\]
and the spectral localization property of $\Phi_k(\Lambda_D)$, we have
\[
\begin{split}
	2^{(s-2)j} \| \psi_j(\Lambda_D)^{\frac{1}{2}} \Delta f \|_{L^p} 
	&\leq C 2^{(s-2)j} \sum_{k \in \mathbb{Z}} 2^{2k} \cdot 
	  \Big\{ 2^{-2j+2k} (1 + 2^{-2j+2k})^{-2} \Big\}^{\frac{1}{2}} 
	   \| \phi_k(\Lambda_D) f \|_{L^p} \\
	&\leq C \sum_{k \in \mathbb{Z}} 2^{(s-2)(j-k)} \cdot 2^{-|j-k|} \cdot 2^{sk} \| \phi_k(\Lambda_D) f \|_{L^p} .
\end{split}
\] 
By taking the $\ell^q(\mathbb{Z})$ norm and applying Young's inequality for convolutions, we obtain
\[
\left\{ \sum_{j \in \mathbb{Z}} \left( 2^{(s-2)j} \| \psi_j(\Lambda_D)^{\frac{1}{2}} \Delta f \|_{L^p} \right)^q \right\}^{\frac{1}{q}} \leq C \left( \sum_{m \in \mathbb{Z}} 2^{(s-2)m} 2^{-|m|} \right) \| f \|_{\dot{B}^s_{p,q}(A_D)}.
\]
The series on the right-hand side converges provided that $|s-2| < 1$, which completes the proof.
\end{pf}

\begin{lem}\label{lem:0128-4} 
	Let $\psi_j(\Lambda_D)$ be defined by \eqref{0126-1}.  Then 
	\[
	\begin{split}
		&  \| \psi_j(\Lambda _D) \|_{L^2\to L^2} \leq \| \psi_j(\Lambda _D) ^{\frac{1}{2}}\|_{L^2 \to L^2},
\\
&	 \| \nabla \psi_j(\Lambda _D) \|_{L^2\to L^2} \leq C 2^j\| \psi_j(\Lambda _D) ^{\frac{1}{2}}\|_{L^2 \to L^2}, 
\\
&  \|  \psi_j(\Lambda _D)^{\frac{1}{2}}\nabla \|_{L^2\to L^2} \leq C 2^j , 
\\
& \| \nabla (1+ 2^{-2j} A _D)^{-1} \|_{L^2 \to L^2} 
\leq C \| \psi _j(\Lambda _D)^{\frac{1}{2}} \|_{L^2 \to L^2}  .
\end{split}
\]
\end{lem}

\begin{pf}
	It is straightforward that the operator norm is estimated by
	\[
	\| \psi_j(\Lambda_D)^{\frac{1}{2}} \|_{L^2 \to L^2} \leq 1,
	\]
	which proves the first inequality. To establish the second inequality, we write
	\[
	\nabla \psi_j(\Lambda_D) = 2^j \cdot \frac{3}{4} \left( 2^{-2j} \nabla A_D^{\frac{1}{2}} (1 + 2^{-2j-2} A_D)^{-1} \right) \left( 2^{-j} A_D^{\frac{1}{2}} (1 + 2^{-2j} A_D)^{-1} \right).
	\]
	We then note the uniform boundedness of the involved operators:
	\[
\sup _{j \in \mathbb Z} 
\|  2^{-2j} \nabla  A_D^{\frac{1}{2}} (1+2^{-2j-2}A_D)^{-1}  \|_{L^2 \to L^2} 
< \infty ,
	\]
	and the norm equivalence:
	\begin{equation} \label{0209-1}
		\| 2^{-j} A_D^{\frac{1}{2}} (1 + 2^{-2j} A_D)^{-1} \|_{L^2 \to L^2} \simeq \| \psi_j(\Lambda_D)^{\frac{1}{2}} \|_{L^2 \to L^2}.
	\end{equation}
	This yields the second inequality.
	
	The third inequality is ensured by a duality argument combined with the second inequality. Finally, the fourth inequality follows from the equivalence $\| \nabla f \|_{L^2} \simeq \| A_D^{\frac{1}{2}} f \|_{L^2}$ for $f \in D(A_D) \subset L^2$ and the relation \eqref{0209-1}.
\end{pf}

\subsection{Independence from boundary conditions for low regularity}

\begin{lem}\label{lem:0127-1}
Suppose that 
\[
1 \leq p < \infty, s < \frac{1}{p} , 1 \leq q \leq \infty, \quad \text{ or } \quad 
1 < p < \infty ,s = \frac{1}{p}, 1 < q  < \infty . 
\]		
Then,  $C^\infty (\overline{\Omega})$ is included in $\dot B^s_{p,q}(A_D)$, where $\overline{\Omega}$ is the closure of $\Omega$. 
\end{lem}

\noindent 
{\bf Remark. } 
The regularity condition $s \leq 1/p$ is motivated by the trace theorem and the denseness of $C_0^\infty(\Omega)$ in $\dot{B}^s_{p,q}(A_D)$ (cf. \cite[Section 2.9.4]{Trieb_1995}). This property can be verified in a similar manner as for Besov spaces defined via the restriction of functions from the whole space. For the critical case $s = 1/p$, the proof of the denseness requires the theorem of S. Mazur (see~\cite[Chapter~5, Section~1, Theorem~2]{Yosi_1980}).

\vskip2mm 

\begin{pf}  
Let $f \in C ^\infty (\overline{\Omega})$. 
It is sufficient to study the case where $s > 0$, since $\dot B^s_{p,q}(A_D) \hookrightarrow \dot B^{s_0}_{p,q}(A_D)$ for $s \leq 0$ due to the embedding property in Proposition~\ref{prop:0128-2}~(4). 

\vskip2mm 
	
\noindent 	Step 1. 
	We introduce a smooth cut-off function $\chi_\delta$ ($0 < \delta \leq 1$) with compact support, satisfying:
\[
\delta < \mathrm{dist}(\mathrm{supp} \, \chi_\delta, \partial \Omega) < 2\delta, \quad \chi_\delta(x) = 1 \text{ if } \mathrm{dist}(x, \partial \Omega) > 2\delta,
\]
and
\[
|\nabla^\alpha \chi_\delta| \leq C \delta^{-|\alpha|} \quad \text{for } |\alpha| \leq 2,
\]
where $C$ is a positive constant independent of $\delta$. The existence of such $\chi_\delta$ is ensured by the compactness of the boundary $\partial \Omega$. 

We now claim that 
\begin{equation}\label{0210-1}
\sup _{0 < \delta \leq 1} \| \chi_{\delta} \|_{\dot B^{\frac{1}{p}}_{p,1}(A_D)} < \infty .
\end{equation}
We choose $j_0$ such that $2^{j_0} \leq \delta ^{-1} \leq 2^{j_0+1}$. 
For the sum over $j \geq j_0$, we utilize the identity 
$\chi _{\delta} = A_D^{-1}A_D\chi _\delta $ and write 
\[
\begin{split}
\sum _{j \geq j_0} 2^{\frac{1}{p}j} \| \phi_j(\Lambda _D) \chi_\delta \|_{L^p}
=& 
\sum _{j \geq j_0} 2^{\frac{1}{p}j} \| A_D ^{-1}\phi_j(\Lambda _D) A_D\chi_\delta \|_{L^p}
\\
\leq 
& C \sum _{j \geq j_0} 2^{(\frac{1}{p}-2)j} \| \Delta \chi _\delta  \|_{L^p}
\\
\leq 
& C 2^{(\frac{1}{p}-2)j_0} \delta ^{-2 + \frac{1}{p}}.
\end{split}
\]
By the definition of $j_0$, we see that the above estimate is uniform with respect to $\delta$. Indeed, 
\[
C 2^{(\frac{1}{p}-2)j_0} \delta ^{-2 + \frac{1}{p}}
\leq C (2 \delta )^{-\frac{1}{p}+2} \delta ^{-2+\frac{1}{p}}
= C 2^{-\frac{1}{p}}. 
\]
For the sum over $j < j_0$, we introduce a smooth cut-off function $\widetilde{\chi}_\delta$ whose support is contained in a neighborhood of $\mathrm{supp} \, \Delta \chi_\delta$.
\[
\widetilde{\chi}_{\delta} := \chi_{\frac{\delta}{4}} - \chi_{4\delta}. 
\]
We note that 
\[
\Delta \chi _\delta = \Delta (\chi _\delta \widetilde{\chi}_\delta), 
\quad 
\phi_j(\Lambda _D)\chi_\delta 
= A_D^{-1} \phi_j(\Lambda_D) A_D \big( \chi_\delta \widetilde{\chi}_\delta \big)
= \phi_j(\Lambda_D) \big(\chi_\delta \widetilde{\chi}_\delta \big). 
\]
The sum over $j < j_0$ is estimated by: 
\[
\sum _{j < j_0} 2^{\frac{1}{p}j} \| \phi_j(\Lambda) \chi_\delta \|_{L^p}
= \sum _{j < j_0} 2^{\frac{1}{p}j} \| \phi_j(\Lambda) \big( \chi_\delta  \widetilde{\chi}_\delta \big) \|_{L^p}
\leq C 2^{\frac{1}{p}j_0} \| \chi_\delta \widetilde{\chi}_\delta \|_{L^p}
\leq C (\delta^{-1})^{\frac{1}{p}} \cdot \delta^{\frac{1}{p}}
= C,
\]
which is a uniform estimate with respect to $\delta\in (0,1]$. 
We obtain \eqref{0210-1}.

\vskip2mm 

\noindent 
Step 2. 
	It is clear that $\chi_\delta f $ and $ A_D(\chi_\delta f) \in L^p$ for $\delta > 0$. 
		We note that the topology of $L^p$ is weaker than that of $\dot{B}^{\frac{1}{p}}_{p,1}(A_D)$, and $\chi_\delta f$ converges to $f$ in $L^p$ as $\delta \to 0$. 
	
	We now claim the following uniform boundedness:
	\begin{equation} \label{0129-1}
		\sup_{0 < \delta \leq 1} \| \chi_\delta f \|_{\dot{B}^{\frac{1}{p}}_{p,q}(A_D)} < \infty.
	\end{equation}
	It is sufficient to show the case when $q=1$ due to the embedding property 
	$\dot B^{\frac{1}{p}}_{p,1}(A_D) \hookrightarrow \dot B^{\frac{1}{p}}_{p,q}(A_D)$. 
	Using the resolution of identity, we decompose $\chi_\delta f$ as follows:
\[
\chi_\delta f = \left( \sum_{k \leq 0} \phi_k(\Lambda_D) \chi_\delta \right) f + \left( \sum_{k > 0} \phi_k(\Lambda_D) \chi_\delta \right) f 
=: (S_0 \chi_\delta ) f + \big((1 - S_0  ) \chi_\delta \big)f.
\]

For the low spectral component of $\chi_\delta$, we utilize the embedding into higher regularity spaces to establish uniform boundedness. Specifically, we have
\[
\begin{split}
	& \| (S_0 \chi_\delta) f \|_{\dot{B}^{\frac{1}{p}}_{p,1}(A_D)}\\
	&\leq C \| A_D ((S_0 \chi_\delta) f) \|_{L^p} \\
	&\leq C \left( \| \Delta (S_0 \chi_\delta) \|_{L^\infty} + \| \nabla (S_0 \chi_\delta) \|_{L^\infty} + \| S_0 \chi_\delta \|_{L^\infty} \right) 
	\left( \| \Delta f \|_{L^p} + \| \nabla f \|_{L^p} + \| f \|_{L^p} \right) \\
	&\leq C \| \chi_\delta \|_{L^\infty} \\
	&\leq C,
\end{split}
\]
where the constant $C$ depends on $f$ but is independent of $\delta \in (0, 1]$. In the third inequality, we used the fact that the derivatives of $S_0 \chi_\delta$ are bounded in $L^\infty$ independently of $\delta$ due to the spectral localization of $S_0$.

For the high spectral component, we divide the sum into two cases: $k \geq j$ and $k < j$. 
In the case where $k \geq j$ and $k \geq 1$, we utilize the boundedness of $\phi_j(\Lambda_D)$ and sum over $j \leq k$ as follows:
\[
\begin{split}
	\sum _{j \in \mathbb Z} \sum _{k \geq \max \{ j, 1 \}  }2^{\frac{1}{p}j} 
	\Big\| \phi_j(\Lambda _D) \Big( \big( \phi_k(\Lambda_D) \chi _\delta \big) f\Big)   \Big\|_{L^p}
	\leq 
	&C  \sum _{k \in \mathbb Z} \sum _{j \leq k }2^{\frac{1}{p} j} 
	\Big\|  \big( \phi_k(\Lambda_D) \chi _\delta \big) f  \Big\|_{L^p}
	\\
	\leq 
	&C \| \chi _\delta  \|_{\dot B^{\frac{1}{p}}_{p,1}(A_D)} \| f \|_{L^\infty}. 
\end{split}
\]
We note that the family $\{ \chi_\delta \}_{0 < \delta \leq 1}$ is uniformly bounded in $\dot{B}^{\frac{1}{p}}_{p,1}(A_D)$ with respect to $\delta$ by Step 1. 
When $1 \leq k < j$, we write 
\[
\phi_j(\Lambda _D) \Big(  \big( \phi_k (\Lambda _D)  \chi_\delta \big)  f\Big)
=A_D^{-1} \phi_j(\Lambda _D) (-\Delta )\Big(  \big( \phi_k (\Lambda _D)  \chi_\delta \big)  f\Big),
\]
and 
\[
\begin{split}
	& 
	\Big\|  \phi_j(\Lambda _D) \Big(  \big( \phi_k (\Lambda _D)  \chi_\delta \big)  f\Big)  \Big\|_{L^p}
	\\
	\leq 
	&C 2^{-2j} (2^{2k} + 2^k + 1) \| \phi_k(\Lambda _D) \chi_\delta \| _{L^p}   
	( \| \Delta f \|_{L^\infty } + \| \nabla f \|_{L^\infty}  +   \|  f  \|_{L^\infty})
	\\
	\leq 
	& C 2 ^{-2j} \cdot 2^{2k}  \| \phi_k(\Lambda _D) \chi_\delta \| _{L^p} , \quad k \geq 1,
\end{split}
\]
where $C$ depends on $f$. 
It follows from the above inequality that 
\[
\sum _{j \in \mathbb Z} \sum _{1 \leq k < j} 2^{\frac{1}{p}j}
\Big\|  \phi_j(\Lambda _D) \Big(  \big( \phi_k (\Lambda _D)  \chi_\delta \big)  f\Big) \Big\|_{L^p}
\leq C \| \chi_{\delta} \|_{\dot B^{\frac{1}{p}}_{p,1}(A_D)}. 
\]
We also note the uniformity of $\| \chi_\delta \|_{\dot B^{\frac{1}{p}}_{p,q}(A_D)}$ with respect to $\delta$, and then obtain \eqref{0129-1}. 

\vskip2mm 

\noindent 
Step 3. We prove that $f \in \dot B^s_{p,q}(A_D)$. 

We first consider the case where $1 <p, q < \infty$ and $s=1/p$. 
Since $\dot B^{\frac{1}{p}}_{p,q}(A_D)$ is a reflexive Banach space, a subsequence $\{\chi_{\delta _k} f\}_{k \in \mathbb N}$, with $\delta _k \to 0$ as $k \to \infty$, exists such that it converges to $f$ weakly in $\dot B^{\frac{1}{p}}_{p,q}(A_D)$ as $N \to \infty$. 
We then conclude $f \in \dot B^{\frac{1}{p}}_{p,q}(A_D)$. 
%We also apply the theorem of S. Masur (see Yosida~\cite[Chapter 5, Section 1, Theorem 2]{Yosi_1980}) and find suitable convex linear combinations of $\{ \chi_{\delta_k} f \}_{k=1}^\infty$, which converges to strongly in $\dot B^{\frac{1}{p}}_{p,q}(A_D)$. 

For the case where $s < 1/p$, $1 < p < \infty$ and $1 \leq q \leq \infty$, the embedding $\dot B^{\frac{1}{p}}_{p,2}(A_D) \hookrightarrow \dot B^s_{p,q}(A_D)$ (see Proposition~\ref{prop:0128-2} (4)) and the previous case ensure that $f \in \dot B^s_{p,q}(A_D)$. 

For the case where $0< s < 1$ and $p = 1$, 
we choose $p_0$ such that $1 < p_0 < \infty$ and 
$s < 1/p_0$. We note the embedding 
$\dot B^s_{p_0,q}(A_D) \hookrightarrow \dot B^s_{1,q}(A_D)$, 
proved by the H\"older inequality and the boundedness of $\Omega$. 
We know $f $ belongs to $ \dot B^{s }_{p_0,q}(A_D)$ 
by the previous case, which proves that 
$
f \in \dot B^{s }_{p_0,q}(A_D) 
\subset  \dot B^s_{1,q}(A_D) $. 
\end{pf}

\vskip3mm

\subsection{Bilinear Estimates}

\begin{prop}\label{prop:1121-1}
Let $\alpha , \beta \in \mathbb N$, $0 < \gamma < 1/2 $, 
  $ 1 \leq p_1, p_2, p_3, p_4 \leq \infty$, and  $1/2 = 1 / p_1 + 1/p_2 = 1/p_3 + 1/p_4 $. 
Then a positive constatnt $C$ exists such that for every $f \in B^{\alpha+\gamma}_{p_1,1}(A_D) \cap \dot B^{\alpha}_{p_3,1}(A_D)$ 
and $g \in \dot B^{\beta}_{p_2,1}(A_D) \cap \dot B^{\beta+\gamma}_{p_4,1}(A_D)) $, 
the function $\big(\nabla ^\alpha  f\big)   \big( \nabla^\beta g \big) $ belongs to the domain of the operator $\Lambda _D^\gamma$ 
defined on $\dot B^0_{2,1}(A_D) $ and  
\[
	\|  \big(\nabla ^\alpha  f\big)   \big( \nabla^\beta g \big) 
	\|_{\dot B^\gamma_{2,1}(A_D)} 
	\leq C \Big( \| f \|_{B^{\alpha+\gamma}_{p_1,1}(A_D)} \| g \|_{\dot B^\beta_{p_2,1}(A_D)}  
	      + \| f \|_{\dot B^\alpha_{p_3,1}(A_D)} \| g \|_{\dot B^{\beta+\gamma}_{p_4,1}(A_D)}  \Big). 
\] 
\end{prop}

\noindent {\bf Remark. } 
The higher order derivatives break the Dirichlet boundary condition of functions in general, 
and this implies that $\nabla ^\alpha f  \nabla ^\beta g$ does not satisfy that condition 
even if $f, g \in \dot B^s_{p,1}(A_D)$ for all $s \geq 0$. On the other hand, it is possible to apply $\Lambda_D^\gamma$ 
for small $\gamma$ due to Lemma~\ref{lem:0127-1}. 
The case $\gamma = 1/2$ can be regarded as the threshold regularity for the validity of the bilinear estimates, particularly in the absence of the Dirichlet boundary condition. For further details, we refer to the monograph by Triebel \cite{Trieb_1995} regarding the trace operator and to \cite{Iw-2023} for the specific bilinear estimates.

\vskip2mm 

\begin{pf}We introduce the notations $f_k, g_l$, where each frequency is restricted to around a dyadic number, 
	and write 
	\[ f = \sum _{k \in \mathbb Z}  \phi_k(\Lambda _D) f= \sum _{k \in \mathbb Z} f_k, \quad 
	  g = \sum _{l \in \mathbb Z} \phi_l(\Lambda _D) f = \sum _{l \in \mathbb Z} g_l .
	\]
We know $\nabla ^\alpha f_k, \nabla ^\beta g_l$ bellong to $C^\infty (\overline{\Omega})$, 
since $f_k, g_l$ belong to the domain of $A_D ^M$, 
the operator on $L^2$, for all $M \in \mathbb N$. 
Lemma~\ref{lem:0127-1} ensures that	$\Lambda _D^\gamma  \big(  \phi_k(\Lambda _D) f \phi_l (\Lambda _D) g \big)$ makes sense.

For each $j$ we write 
\[
\begin{split}
\phi_j(\Lambda _D)
\Lambda _D^\gamma (\nabla ^\alpha f \nabla ^\beta g) 
= 
&
\Big( \sum _{k>j, l \in \mathbb Z} + \sum_{k \in \mathbb Z, l>j}  + \sum _{k\leq j, l \leq j}
\Big)
 \phi_j(\Lambda _D)\Lambda _D^\gamma (\nabla ^\alpha f_k  \nabla^\beta g_l)
\\
=& I_j + II_j + III_j .
\end{split}
\]
For the term $I_j$, it follows from $j < k$ that 
\[
\sum _{ j \in \mathbb Z}\| I_j \|_{L^2} 
\leq C \sum _{ j \in \mathbb Z}
  2^{\gamma j } 
  \sum _{k > j} \| \nabla ^\alpha f_k \|_{L^{p_1}}\| \nabla ^\beta g_l \|_{L^{p_2}} 
\leq C \| f \|_{\dot B^{\gamma + \alpha}_{p_1,1}(A_D)}\| g \|_{\dot B^\beta_{p_2 ,1}(A_D)}.
\]
Similarly, 
\[
\sum _{j \in \mathbb Z} \| II_j \|_{L^2} 
\leq C 
  \| f \|_{\dot B^\alpha_{p_3,1}(A_D)} \| g \|_{\dot B^{ \beta +\gamma}_{p_4,1}(A_D)} . 
\]

For the third term $III_j$, we note that 
$\displaystyle \sum _{j \in \mathbb Z} \sum _{k,l \leq j} 
= \sum _{k,l \in \mathbb Z} \sum _{j \geq \max\{k,l\}}$ 
and study the sum over $j$ for each $k,l$. 
We utilize the smooth function $\chi _\delta \, (0 < \delta \leq 1)$ introduced in the proof of Lemma~\ref{lem:0127-1}, and know that the approximation with compact support functions is available in $\dot B^\gamma_{2,1}(A_D)$. 
We see from the Fatou Lemma that for every $k$ and $l$
\[
\sum _{j \geq \max \{ k,l \}}\| \phi_j(\Lambda _D) \Lambda _D ^\gamma (\nabla^\alpha  f_k \nabla ^\beta g_l) \|_{L^2} 
\leq  \liminf_{\delta \to 0} \sum _{j\geq \max\{ k,l\}} \| \phi_j(\Lambda _D) \Lambda _D^\gamma (\chi _\delta \nabla^\alpha  f_k \nabla ^\beta g_l) \|_{L^2} .
\]
We also note the following convergence: 
\[
\| \chi_\delta  \|_{\dot B^\gamma_{2,1}(A_D)} \leq C \delta ^{-\gamma + \frac{1}{2}},
\]
which is obtained by the same argument as in the proof Step 1 of Lemma~\ref{lem:0127-1} with the regularity $\gamma$ instead of $1/p$. 

We decompose $\chi_\delta$ with high and low spectral components. 
\[
\chi _\delta = \sum _{i \leq j} \phi_i(\Lambda _D) \chi_\delta 
+ \sum _{i > j} \phi_i (\Lambda _D) \chi _\delta
=: S_j  \chi _\delta + (1-S_j) \chi _{\delta }. 
\]
For the high spectral component, we simly estimate 
\[
\begin{split}
&	\sum _{j \geq \max\{ k,l\}} \Big\| \phi_j(\Lambda _D) \Lambda _D^\gamma  
  \Big(  ( (1-S_j)\chi _\delta ) \nabla^\alpha  f_k \nabla ^\beta g_l \Big) \Big\|_{L^2}
\\
\leq 
& 
C \sum _{j \in \mathbb Z}2^{\gamma j} 
    \sum _{i > j} \| \phi_i (\Lambda _D) \chi _\delta  \|_{L^2}
     \| \nabla ^\alpha f_k \|_{L^\infty} \| \nabla ^\beta g_l \|_{L^\infty}
\\
\leq 
& 
C \| \chi _\delta  \|_{\dot B^{\gamma}_{2,1}(A_D)}
\| \nabla ^\alpha f_k \|_{L^\infty} \| \nabla ^\beta g_l \|_{L^\infty}
\\
\leq 
& 
C \delta ^{-\gamma + \frac{1}{2}}
\| \nabla ^\alpha f_k \|_{L^\infty} \| \nabla ^\beta g_l \|_{L^\infty} 
\\
\to 
& 0 \quad \text{as } \delta \to 0, \quad \text{ for each } k, l. 
\end{split}
\]
We turn to consider the low spectral part $S_j \chi _\delta $. 
Since $(S_j \chi _\delta) \nabla ^\alpha f_k \nabla ^\beta g_l$ belongs to $C^\infty (\Omega)$, vanishes on the boundary, and $A_D=-\Delta $ is able to act on the term, we write 
\[
\phi_j(\Lambda _D) \Lambda _D ^\gamma \Big( (S_j \chi _\delta) \nabla^\alpha  f_k \nabla ^\beta g_l\Big) 
= \Lambda _D^{-2+\gamma} \phi_j(\Lambda _D) (-\Delta) \Big( (S_j \chi _\delta)  \nabla^\alpha  f_k \nabla ^\beta g_l \Big) . 
\]
and 
\begin{equation}\label{0304-1}
\begin{split}
	& 
	\sum _{j \geq \max\{k,l\}}\| \phi_j(\Lambda _D) \Lambda _D ^\gamma ( (S_j \chi_\delta )\nabla^\alpha  f_k \nabla ^\beta g_l) \|_{L^2} 
\\
\leq 
& C  \sum _{j \geq \max\{k,l\}}2^{(-2 + \gamma )j} 
\| \phi_j(\Lambda _D) (-\Delta )((S_j \chi _\delta ) \nabla^\alpha  f_k \nabla ^\beta g_l) \|_{L^2} .
\end{split}
\end{equation}
We apply the Leibniz rule to $-\Delta$ in the right-hand-side of \eqref{0304-1} 
and the term with $-\Delta$ acting on the product $\nabla ^\alpha f_k \nabla ^\beta g_l$ 
is estimated as follows. 
\[
\begin{split}
&\sum_{k,l \in \mathbb Z}	\liminf_{\delta \to 0} \sum _{j \geq \max\{k,l\}}2^{(-2 + \gamma )j} 
\| \phi_j(\Lambda _D) (S_j \chi _\delta )  (-\Delta )(\nabla^\alpha  f_k \nabla ^\beta g_l) \|_{L^2}
\\
\leq 
&\sum_{k,l \in \mathbb Z}  \sum_{j \geq \max\{k,l\}} 2^{(-2+\gamma)j} (2^{2k} + 2^{2l}) \cdot 2^{\alpha k} \| f_k  \|_{L^{p_1}} \cdot 2^{\beta l}\|  g_l \|_{L^{p_2}}
\\
\leq 
& C \| f \|_{\dot B^{\alpha + \gamma}_{p_1,1}(A_D)} \| g \|_{\dot B^{\beta}_{p_2,1}(A_D)} . 
\end{split}
\]
Except for the above term, we have 
\[
\begin{split}
& 	 \sum _{j \geq \max\{k,l\}}2^{(-2 + \gamma )j} 
\Big\| \phi_j(\Lambda _D) 
   \Big((-\Delta )S_j \chi _\delta )  \nabla^\alpha  f_k \nabla ^\beta g_l   
        -2 \nabla S_j \chi_\delta \cdot \nabla \big(\nabla^\alpha  f_k \nabla ^\beta g_l   \big)
   \Big)
\Big \|_{L^2}
\\
\leq 
&C \sum_{j \in \mathbb Z} 2^{(-2+\gamma)j} \sum _{i \leq j} (2^{2i} + 2^i) \| \phi_i(\Lambda _D) \chi_\delta  \|_{L^2}
     \Big(  \| \nabla^\alpha  f_k \nabla ^\beta g_l  \|_{L^\infty}
         +     \| \nabla \big(\nabla^\alpha  f_k \nabla ^\beta g_l   \big) \|_{L^\infty}
     \Big)
\\
\leq 
&C (\| \chi_\delta \|_{\dot B^{\gamma}_{2,1}(A_D)}  + \| \chi_\delta \|_{\dot B^{-1+\gamma}_{2,1}(A_D)}) 
\\
\leq 
&C \| \chi_\delta \|_{\dot B^{\gamma}_{2,1}(A_D)}  
\to 0 \quad \text{as } \delta \to 0, 
\quad \text{ for each } k,l, 
\end{split}
\]
where $C$ in the last line depends on $f_k, g_l$, 
and we have applied  the embedding $\dot B^\gamma_{2,1}(A_D) \hookrightarrow \dot B^{-1+\gamma}_{2,1}(A_D)$ due to the positivity of the spectrum and the estimate $\| \chi_\delta \|_{\dot B^\gamma_{2,1}(A_D)} \leq C \delta ^{-\gamma + \frac{1}{2}}$. 
We then conclude that 
\[
\begin{split}
	\sum _{ j \in \mathbb Z}\| III_j \|_{L^2} 
\leq & \sum _{j \in \mathbb Z} \lim _{\delta \to 0}
\sum _{j\geq \max\{ k,l\}} \| \phi_j(\Lambda _D) \Lambda _D^\gamma (\chi _\delta \nabla^\alpha  f_k \nabla ^\beta g_l) \|_{L^2}
\\
\leq 
& C \| f \|_{\dot B^{\gamma + \alpha}_{p_1,1}(A_D)}\| g \|_{\dot B^\beta_{p_2 ,1}(A_D)}.
\end{split}
\]
We complete the proof. 
\end{pf}

\section{Proof of Theorem \ref{thm:1}}

Instead of the original nonlinear term, we first consider the regularized term $N_\mu$ defined by 
\begin{equation}\label{1126-3}
	N_\mu (\theta, \widetilde \theta) := (1+\mu A_D)^{-1} \left\{ \left(\nabla^\perp \Lambda_D^{-1} \theta\right) \cdot \nabla (1+ \mu A_D)^{-1}\widetilde \theta \right\}, \quad \mu > 0.
\end{equation}
Lemma~\ref{lem:1126-5} establishes the necessary inequalities to handle this nonlinear term. In Proposition~\ref{prop:1010-5}, we construct a solution to the modified equation:
\[
\partial_t \theta + N_\mu (\theta, \theta) = 0.
\]
In particular, for initial data in $\dot{B}^2_{2,1}(A_D)$, we find an existence time $T > 0$, independent of $\mu > 0$, for the solution in $C([0,T],  \dot{B}^2_{2,1}(A_D))$. Furthermore, we establish the uniform boundedness of the solution in a function space of the Chemin--Lerner type. Finally, by considering the limit of $\theta_\mu$ as $\mu \to 0$, we construct a solution to the governing equation $\partial_t \theta + N_0(\theta, \theta) = 0$, satisfying the prescribed regularity and uniqueness properties. We note that the uniform estimates are also instrumental in proving the continuity of the solution with respect to time.

\begin{lem}\label{lem:1126-5}
	Let $N_\mu$ be defined by \eqref{1126-3}. 
Suppose that $\theta \in \dot B^2_{2,1}(A_D)$ and $\theta \not \equiv 0$. 
Then 
\begin{equation}\label{1118-1}
\sum _{j \in \mathbb Z} 
\Big| \int _{\Omega} \Big( \Delta  N_\mu (S_j\theta , \theta) \Big) \psi _j (\Lambda _D) \Delta \theta ~dx \Big| 
\cdot \dfrac{1}{ \| \psi_j (\Lambda _D)^{\frac{1}{2}}  \Delta \theta \|_{L^2}}
\leq C \| \theta  \|_{\dot B^2_{2,1}(A_D)} ^2, 
\end{equation}
\begin{equation}\label{1118-2}
\sum _{j \in \mathbb Z} 
\Big|  \int _{\Omega} \Big( \Delta  N_{\mu } ((1-S_j)\theta , \theta) \Big) \psi _j (\Lambda _D) \Delta \theta ~dx 
\Big| 
\cdot \dfrac{1}{\| \psi_j(\Lambda _D)\Delta \theta \|_{L^2}} 
\leq C \| \theta  \|_{\dot B^2_{2,1}(A_D) } ^2 ,
\end{equation}
where $C$ is a positive constant independent of $\mu$ and $\theta$. 
\end{lem}

\begin{pf}
	Step 1. 
We prove the first inequality \eqref{1118-1}. Let us introduce the notation
\[
\theta_\mu := (1 + \mu A_D)^{-1} \theta.
\]
We can write
\[
\Delta_D N_\mu(S_j \theta, \theta) = \Delta_D (1 + \mu A_D)^{-1} N_0(S_j \theta, \theta_\mu) = (1 + \mu A_D)^{-1} \Delta_D N_0(S_j \theta, \theta_\mu),
\]
since $N_0(S_j \theta, \theta_\mu)$ belongs to the domain of the Dirichlet Laplacian $D(A_D) \subset L^2$. By further utilizing the self-adjointness (symmetric property) of the resolvent $(1 + \mu A_D)^{-1}$ on $L^2$, we find that the inequality \eqref{1118-1} is equivalent to
\begin{equation}\label{1120-1}
\sum _{j \in \mathbb Z} 
\Big| \int _{\Omega} \Big( \Delta _D N_0 \big(  S_j \theta,  \theta _\mu \big) \Big)
\psi _j (\Lambda _D)\Delta  \theta _\mu ~dx 
\Big| \cdot \dfrac{1}{\| \psi_j (\Lambda _D)^{\frac{1}{2}}  \Delta \theta \|_{L^2}}
 \leq C \| \theta \|_{\dot{B}^2_{2,1}(A_D)}^2.
\end{equation}
We will estimate the operator norm of the resolvent $(1+\mu A_D)^{-1}$ on $L^2$. 
\[
\| (1+\mu A_D)^{-1} \|_{L^2 \to L^2} \leq 1, \quad 
\| \theta_\mu \|_{L^2} \leq \| \theta \|_{L^2},
\]
which hold uniformly with respect to $\mu > 0$. Using the elementary identity
\[
\Delta(fg) = f \Delta g - (\Delta f) g + 2\nabla \cdot ((\nabla f) g),
\]
derived from the Leibniz rule, we can decompose the term $\Delta_D N_0(S_j\theta, \theta_\mu)$ as follows:
\[
\begin{split}
	\Delta_D N_0(S_j\theta, \theta_\mu) 
	&= ( \nabla^\perp \Lambda_D^{-1} S_j \theta \cdot \nabla ) \Delta \theta_\mu 
	- ( \nabla^\perp \Lambda_D^{-1} \Delta S_j \theta \cdot \nabla ) \theta_\mu \\
	&\quad + 2 \sum_{i=1}^2 \partial_{x_i} \left( ( \partial_{x_i} \nabla^\perp \Lambda_D^{-1} S_j \theta \cdot \nabla ) \theta_\mu \right) \\
	&=: I_1(\theta_\mu) + I_2(\theta_\mu) + 2 I_3 (\theta_\mu).
\end{split}
\]

\noindent 
\underline{Estimate of $I_1$}. 
By the definition \eqref{0126-1} of $\psi_j(\Lambda_D)$, we have
\[
\begin{split}
	\int_{\Omega} I_1(\theta_\mu) \psi_j(\Lambda_D) \Delta \theta_\mu \, dx 
	&= \sum_{l=0}^1 (-1)^l \int_{\Omega} I_1(\theta_\mu) (1 - 2^{-2j-2l} \Delta_D)^{-1} \Delta \theta_\mu \, dx.
\end{split}
\]
To handle the integrand, we apply the following identity to 
$I_1(\theta _\mu)$. 
\[\begin{split}
	\theta _\mu 
	=& (1-2^{-2j-2l}\Delta _D) (1-2^{-2j-2l} \Delta_D)^{-1}  \theta _\mu
	\\
	=&  \Big(1-2^{-2j-2l}  \sum _{i=1}^2 \partial _{x_i}^2 \Big) (1-2^{-2j-2l} \Delta_D)^{-1} \theta _\mu. 
\end{split}
\] 
Except for the term involving the second derivative $\partial_{x_i}^2$, we observe that 
\[
\int _{\Omega}  \Big(  (\nabla ^\perp \Lambda _D^{-1} S_j \theta \cdot \nabla)  \Theta_{j+l}  \Big) \cdot \Theta_{j+l}
~dx = 0, 
\quad \text{with } \Theta_{j+l} := (1- 2^{-2j-2l} \Delta_D)^{-1} \Delta \theta _\mu. 
\]
Using this property, we write
\[
\begin{split}
	\int _{\Omega} I_1 (\theta _\mu) \psi _j (\Lambda _D) \Delta \theta_\mu  ~dx 
	=& -\sum _{l=0}^1 (-1)^l  
	   \int _{\Omega}  \Big(  (\nabla ^\perp \Lambda _D^{-1} S_j \theta \cdot \nabla) 2^{-2j-2l} \Delta \Theta_{j+l}  \Big) \cdot \Theta_{j+l}
	~dx . 
\end{split}
\]
For each $l$, we apply integration by parts twice to obtain
\[
\begin{split}
& 	   \int _{\Omega}  \Big(  (\nabla ^\perp \Lambda _D^{-1} S_j \theta \cdot \nabla) 2^{-2j-2l} \Delta \Theta_{j+l}  \Big) \cdot \Theta_{j+l}
~dx
\\
=& 2^{-2j-2l}  \int _{\Omega}  
\Big\{ \Big( ( \Delta \nabla ^\perp \Lambda _D^{-1} S_j \theta \cdot \nabla)   \Theta_{j+l}  \Big)   \Theta_{j+l}
      +  2 \Big( ( \nabla \nabla ^\perp \Lambda _D^{-1} S_j \theta \cdot \nabla)   \Theta_{j+l} \Big)   \nabla  \Theta_{j+l}
\\
&     + \Big( (  \nabla ^\perp \Lambda _D^{-1} S_j \theta \cdot \nabla)   \Theta_{j+l} \Big)   \Delta \Theta_{j+l}
    \Big\}
 ~dx 
\end{split}
\]
Applying integration by parts to the third integrand yields the same term as on the left-hand side but with a negative sign; thus, we obtain
\[
\begin{split}
	& 	   \int _{\Omega}  \Big(  (\nabla ^\perp \Lambda _D^{-1} S_j \theta \cdot \nabla) 2^{-2j-2l} \Delta \Theta_{j+l}  \Big) \cdot \Theta_{j+l}
	~dx
	\\
	=& \dfrac{1}{2} \cdot 2^{-2j-2l}  \int _{\Omega}  
	\Big\{ \Big( ( \Delta \nabla ^\perp \Lambda _D^{-1} S_j \theta \cdot \nabla)   \Theta_{j+l}  \Big)   \Theta_{j+l}
	+  2 \Big( ( \nabla \nabla ^\perp \Lambda _D^{-1} S_j \theta \cdot \nabla)   \Theta_{j+l} \Big)   \nabla  \Theta_{j+l}
	\Big\}
	~dx . 
\end{split}
\]
We apply the $L^3 \times L^2 \times L^6$ estimate combined with Lemma~\ref{lem:0128-3}~(1) and the embedding $H^1(\Omega) \hookrightarrow L^6$ for the first integrand. For the second integrand, we use the $L^\infty \times L^2 \times L^2$ estimate. These lead to:
\[
\begin{split}
	& 	2^{-2j-2l} \Big|  \int _{\Omega}  \Big(  (\nabla ^\perp \Lambda _D^{-1} S_j \theta \cdot \nabla) 2^{-2j-2l} \Delta \Theta_{j+l}  \Big) \cdot \Theta_{j+l}
	~dx
	\Big| 
	\\
	\leq 
	& C 	2^{-2j-2l} 
	  \Big( \| S_j \theta \|_{\dot B^{2}_{2,1}(A)} \| \nabla \Theta _{j+l} \|_{L^2} \| \Theta _{j+l} \|_{H^1(\Omega)}
	      + \| S_j \theta \|_{\dot B^{2}_{2,1}(A)} \| \nabla \Theta _{j+l} \|_{L^2}^2 
	   \Big) .
\end{split}
\]
It follows from the fourth inequality in Lemma~\ref{lem:0128-4} that 
\[
\begin{split}
		 \Big|  \int _{\Omega}  \Big(  (\nabla ^\perp \Lambda _D^{-1} S_j \theta \cdot \nabla) 2^{-2j-2l} \Delta \Theta_{j+l}  \Big) \cdot \Theta_{j+l}
	~dx
	\Big| 
	\leq 
	 C \|  \theta \|_{\dot B^{2}_{2,1}(A)} \| \psi_j(\Lambda _D)^{\frac{1}{2}} \Delta \theta_\mu \|_{L^2}^2. 
\end{split}
\]
This, combined with the norm equivalence in Lemma~\ref{lem:0128-3} (2), allows us to obtain
\[
\begin{split}
	\sum _{j \in \mathbb Z} \Big| \int _{\Omega} I_1 (\theta _\mu) \psi _j (\Lambda _D) \Delta \theta _\mu ~dx \Big|  
	\cdot \dfrac{1}{\| \psi _j (\Lambda _D)^{\frac{1}{2}} \Delta \theta  \|_{L^2 }}
	\leq 
	C \| \theta \|_{\dot B^2_{2,1}(A_D)} ^2. 
\end{split}
\]

\vskip2mm

\noindent 
\underline{Estimate of $I_2$}. Using the decomposition $\theta_\mu = S_j \theta_\mu + (1 - S_j) \theta_\mu$, we write
\[
\begin{split}
	\int _{\Omega} I_2 (\theta _\mu) \psi _j (\Lambda _D) \Delta \theta_\mu  ~dx 
	=& 	\int _{\Omega} I_2 ( S_j \theta _\mu) \psi _j (\Lambda _D) \Delta \theta_\mu  ~dx 
	+ \int _{\Omega} I_2 ( (1- S_j) \theta _\mu) \psi _j (\Lambda _D) \Delta \theta_\mu  ~dx .
\end{split}
\]
For the first term, let $0 < \gamma < 1/2$. We utilize the duality $(\dot{B}^\gamma_{2,1}(A_D))^* = \dot{B}^{-\gamma}_{2,\infty}(A_D)$. Applying Proposition~\ref{prop:1121-1} to the term $I_2(S_j \theta_\mu)$, we obtain
\[
\begin{split}
&	\Big| \int _{\Omega} I_2 ( S_j \theta _\mu) \psi _j (\Lambda _D) \Delta \theta_\mu  ~dx\Big| 
\\ 
\leq 
&\| I_2 (S_j \theta _\mu ) \|_{\dot B^{\gamma}_{2,1}(A_D)} \| \psi _j (\Lambda _D) \Delta \theta _\mu  \|_{\dot B^{-\gamma} _{2,\infty}(A_D)}
\\
\leq 
& C \Big(\| S_j \theta  \|_{\dot B^{2+\gamma}_{2,1} (A_D)} \| S_j \theta _{\mu } \|_{\dot B^1_{\infty ,1}(A_D)} 
        + \| S_j \theta  \|_{\dot B^2_{2,1}(A_D)}  \| S_j \theta _\mu \|_{\dot B^{1+\gamma}_{\infty,1}(A_D)}
        \Big)
        \|  \Lambda _D^{-\gamma} \psi_j (\Lambda _D) \theta _\mu  \|_{L^2} .
\end{split}
\]
It follows from the embedding $\dot{B}^{s+1}_{2,1}(A_D) \hookrightarrow \dot{B}^s_{\infty,1}(A_D)$ for $s= 1, 1+\gamma$ , together with the uniform boundedness of the resolvent $(1 + \mu A_D)^{-1}$ on $L^2$ with respect to $\mu > 0$, that
\[
\begin{split}
	&	\Big| \int _{\Omega} I_2 ( S_j \theta _\mu) \psi _j (\Lambda _D) \Delta \theta_\mu  ~dx\Big| 
	\\ 
	\leq 
	& C \| S_j \theta  \|_{\dot B^{2+\gamma}_{2,1} (A_D)} \| S_j \theta \|_{\dot B^2_{2 ,1}(A_D)} 
	\|  \Lambda _D^{-\gamma} \psi_j (\Lambda _D) \theta  \|_{L^2} .
	\\ 
\leq 
& C \sum _{k \leq j} 2^{(2+\gamma)k} \| \phi_k(\Lambda _D) \theta  \|_{L^2} \| \theta \|_{\dot B^2_{2 ,1}(A_D)} 
\cdot 2^{-\gamma j}\|  \psi_j (\Lambda _D) ^{\frac{1}{2}}\theta  \|_{L^2} ,
\end{split}
\]
where we have utilized the spectral localization of $S_j$ and the following elementary estimates for operators on $L^2$:
\[
\psi_j(\Lambda _D) \simeq  \Big( 2^{-j} (-\Delta _D)^{\frac{1}{2}} (1- 2^{-2j} \Delta _D)^{-\frac{1}{2}} \Big) ^{2}, 
\qquad 
\eta ^{-\gamma} \psi_j (\eta) \leq C 2^{ -\gamma j} \psi _j (\eta)^{\frac{1}{2}}, \quad \eta > 0 . 
\]
We then have 
\[
\begin{split}
	&	\sum_{j \in \mathbb Z} \Big| \int _{\Omega} I_2 ( S_j \theta _\mu) \psi _j (\Lambda _D) \Delta \theta_\mu  ~dx\Big| 
	 \cdot \dfrac{1}{\| \psi_j (\Lambda _D)^{\frac{1}{2}} \Delta \theta \|_{L^2}}
	\\ 
	\leq 
	& C \sum _{j \in \mathbb Z} 2^{-\gamma j} \sum _{k \leq j} 2^{(2+\gamma)k} \| \phi_k(\Lambda _D) \theta  \|_{L^2} \| \theta \|_{\dot B^2_{2 ,1}(A_D)} 
	\\ 
\leq 
& C \| \theta \|_{\dot B^2_{2 ,1}(A_D)} ^2 . 
\end{split}
\]
As for $I_2 ((1-S_j)\theta_\mu)$, we simply estimate the term by the H\"older inequality and the property $\psi_j \leq \psi_j^{1/2}$ in the sense of operator norms.
\[
\begin{split}
	&	\sum_{j \in \mathbb Z} \Big| \int _{\Omega} I_2 ( (1-S_j) \theta _\mu) \psi _j (\Lambda _D) \Delta \theta_\mu  ~dx\Big| 
	\cdot \dfrac{1}{\| \psi_j (\Lambda _D)^{\frac{1}{2}} \Delta \theta \|_{L^2}}
	\\ 
	\leq 
	& \sum _{j \in \mathbb Z}\| \nabla ^\perp \Lambda _D ^{-1} \Delta S_j \theta  \|_{L^\infty}   \| \nabla (1-S_j) \theta _\mu  \|_{L^2  }
	\cdot \| \psi _j (\Lambda_D) \Delta \theta _\mu  \|_{L^2}
	   	\cdot \dfrac{1}{\| \psi_j (\Lambda _D)^{\frac{1}{2}} \Delta \theta \|_{L^2}}
	\\ 
	\leq 
	& C \sum _{j \in \mathbb Z} \sum _{ k \leq j} 2^{3k} \| \phi_k(\Lambda _D) \theta  \|_{L^2} 
	        \sum _{l \geq j }  2^l \| \phi_l(\Lambda _D) \theta _\mu  \|_{L^2}
\\
\leq 
& C \| \theta  \|_{\dot B^2_{2,1}(A_D)}^2 . 
\end{split}
\]

\vskip2mm

\noindent 
\underline{Estimate of $I_3$}. 
By the decomposition $\theta_\mu = S_j\theta_\mu + (1-S_j) \theta_\mu$, we have 
\[
\begin{split}
	\int _{\Omega} I_3 (\theta _\mu) \psi _j (\Lambda _D) \Delta \theta_\mu  ~dx 
	=& 	\int _{\Omega} I_3 ( S_j \theta _\mu) \psi _j (\Lambda _D) \Delta \theta_\mu  ~dx 
        + \int _{\Omega} I_3 ( (1- S_j) \theta _\mu) \psi _j (\Lambda _D) \Delta \theta_\mu  ~dx . 
\end{split}
\]
We first consider the first term. Let $0 < \gamma < 1/2$. 
It follows from Proposition~\ref{prop:1121-1} that $\Lambda_D^\gamma I_3 (S_j \theta_\mu)$ is well-defined, and we can write
\[
\begin{split}
	\int _{\Omega} I_3 ( S_j \theta _\mu) \psi _j (\Lambda _D) \Delta \theta_\mu  ~dx
	= 
	& 	\int _{\Omega} \Big(  \Lambda _D^{\gamma }I_3 ( S_j \theta _\mu) \Big) \Lambda _D^{-\gamma}  \psi _j (\Lambda _D) \Delta \theta_\mu  ~dx. 
\end{split}
\]
Specifically, $\Lambda_D^\gamma I_3 (S_j \theta_\mu)$ is expressed as
\[
\Lambda _D^{\gamma } I_3 ( S_j \theta _\mu) 
=
\Lambda _D^{\gamma } \Big( ( \Delta \nabla ^\perp \Lambda _D^{-1} S_j \theta  \cdot  \nabla ) S_j \theta _\mu \Big) 
+ 
\Lambda _D^{\gamma} \sum _{i=1}^2 \Big( ( \partial _{x_i} \nabla ^\perp \Lambda _D^{-1} S_j \theta  \cdot  \partial _{x_i}\nabla ) S_j\theta _\mu \Big) ,
\]
where we use the spectral localizations
\[
S_j\theta  = \sum _{k \leq j} \phi_k (\Lambda _D) \theta , \qquad 
S_j \theta _\mu = \sum _{l \leq j} \phi_l (\Lambda _D) \theta _\mu .
\]
By applying the bilinear estimate in Proposition~\ref{prop:1121-1} and the embedding $\dot{B}^{s+1}_{2,1}(A_D) \hookrightarrow \dot{B}^s_{\infty,1}(A_D)$ for $s=1$ and $s=1+\gamma$, we estimate the term as follows:
\[
\begin{split}
&	\Big| 	\int _{\Omega} I_3 ( S_j \theta _\mu) \psi _j (\Lambda _D) \Delta \theta_\mu  ~dx
     \Big|
\\
\leq 
& C \| \Lambda _D^{\gamma } I_3 (S_j \theta _\mu) \|_{L^2} 
      \| \Lambda _D^{-\gamma} \psi _j(\Lambda _D) \Delta  \theta _\mu  \|_{L^2}
\\
\leq 
& C \sum _{k, l \leq j} 
    \Big(  \| \phi_k(\Lambda _D) \theta  \|_{\dot B^{2+\gamma}_{2,1}(A_D)} \| \phi_l(\Lambda _D) \theta _\mu \|_{\dot B^2_{2 ,1}(A_D)}   
            + \| \phi_k (\Lambda _D)\theta  \|_{\dot B^2_{2,1}(A_D)} \| \phi_l(\Lambda_D)\theta _\mu  \|_{\dot B^{2+\gamma }_{2,1}(A_D)}  \Big) 
 \\
 & \cdot \Big\| \Lambda _D ^{-\gamma } \psi_j (\Lambda _D) ^{\frac{1}{2}} \Big( \psi_j (\Lambda _D)^{\frac{1}{2}} \Delta \theta _\mu \Big)  \Big\| _{L^2}
\\
\leq 
& C\Big(  \sum _{k, l \leq j} (2^{(2+\gamma)k} 2^{2l} + 2^{2k} 2^{(2+\gamma )l} )
   \| \phi_k(\Lambda _D) \theta  \|_{L^2} \| \phi_l(\Lambda _D) \theta _\mu \|_{L^2}
   \Big) 
  \cdot 2^{-\gamma j}  \| \psi_j (\Lambda _D)^{\frac{1}{2}} \Delta \theta _\mu \|_{L^2}.
\end{split}
\]
By dividing by $\| \psi_j (\Lambda_D)^{\frac{1}{2}} \Delta \theta \|_{L^2}$ and taking the sum over $j \in \mathbb{Z}$, we obtain
\[
\sum _{j \in \mathbb Z} 
\Big| 	\int _{\Omega} I_3 ( S_j \theta _\mu) \psi _j (\Lambda _D) \Delta \theta_\mu  ~dx
\Big|
\cdot \dfrac{1}{\| \psi_j (\Lambda _D)^{\frac{1}{2}} \Delta \theta  \|_{L^2}}
\leq C \| \theta  \|_{\dot B^2_{2,1}(A_D)} ^2 . 
\]

As for $I_3 ((1-S_j) \theta_\mu)$, we apply the symmetry of $\psi_j (\Lambda_D)^{1/2}$ and the boundedness of $\psi_j (\Lambda_D)^{1/2} \nabla$ on $L^2$ with the bound $C 2^j$. We then have
\begin{equation}\label{1126-2}
\begin{split}
&
	\Big| \int _{\Omega} I_3 ( (1- S_j) \theta _\mu) \psi _j (\Lambda _D) \Delta \theta_\mu  ~dx\Big| 
\\
\leq
& \| \psi _j (\Lambda _D) ^{\frac{1}{2}} I_3 ((1-S_j) \theta _\mu ) \|_{L^2} 
      \| \psi _j (\Lambda _D) ^{\frac{1}{2} } \Delta \theta _\mu  \|_{L^2}
\\
\leq 
& C 2^j \sum _{i=1}^2 \| (\partial _{x_i} \nabla ^\perp \Lambda _D^{-1} S_j\theta \cdot \nabla ) (1-S_j)\theta _\mu \|_{L^2} 
      \| \psi _j (\Lambda _D) ^{\frac{1}{2} } \Delta \theta _\mu  \|_{L^2}
\\
\leq 
& C 2^j \sum _{k \leq j} 2^k \| \phi_k(\Lambda _D)\theta \|_{L^\infty} \sum _{l >j} 2^l \| \phi_l(\Lambda _D)\theta _\mu \|_{L^2}
 \cdot \| \psi _j (\Lambda _D) ^{\frac{1}{2} } \Delta \theta _\mu  \|_{L^2}
\\
\leq 
& C 2^j \| \theta \|_{\dot B^2_{2,1}(A_D)} \sum _{l >j} 2^l \| \phi_l(\Lambda _D)\theta _\mu \|_{L^2}
\cdot \| \psi _j (\Lambda _D) ^{\frac{1}{2} } \Delta \theta _\mu  \|_{L^2},
\end{split}
\end{equation}
where we have used the estimate related to the continuous embedding $\dot B^2_{2,1}(A_D) \hookrightarrow \dot B^1_{\infty,1}(A_D)$. 
By dividing by $\| \psi_j (\Lambda_D)^{\frac{1}{2}} \Delta \theta \|_{L^2}$ and taking the sum over $j \in \mathbb{Z}$, we obtain
\[
\begin{split}
	&
		\sum_{j \in \mathbb Z} \Big| \int _{\Omega} I_3 ( (1- S_j) \theta _\mu) \psi _j (\Lambda _D) \Delta \theta_\mu  ~dx\Big| 
		\cdot \dfrac{1}{\| \psi _j(\Lambda _D)^{\frac{1}{2}} \Delta \theta  \|_{L^2}} 
	\\
	\leq 
	& C \| \theta  \|_{\dot B^2_{2,1}(A_D)}  \sum _{j \in \mathbb Z} 2^j \sum _{l > j} 2^l \| \phi_l(\Lambda _D)\theta _\mu  \|_{L^2}
	\\
	\leq 
	& C \| \theta  \|_{\dot B^2_{2,1}(A_D)} ^2 .
\end{split}
\]
This completes the proof of \eqref{1120-1}.

\vskip2mm 

\noindent 
Step 2. 
We prove the second inequality \eqref{1118-2}. Similarly to \eqref{1120-1}, it suffices to show that
\begin{equation}\label{1126-1}
	\sum _{j \in \mathbb Z} 
	\Big| \int _{\Omega} \Big( \Delta _D N_0 \big(  (1-S_j) \theta,  \theta _\mu \big) \Big)
	\psi _j (\Lambda _D)\Delta  \theta _\mu ~dx 
	\Big| \cdot \dfrac{1}{\| \psi_j (\Lambda _D)^{\frac{1}{2}}  \Delta \theta \|_{L^2}}
	\leq C \| \theta  \|_{\dot B^2_{2,1}(A_D)} ^2 . 
\end{equation}
Analogously to the estimate \eqref{1126-2}, we utilize the boundedness of $\psi_j (\Lambda_D)^{1/2} \nabla$ with the bound $C 2^j$. By applying the Leibniz rule to the first-order derivatives, the H\"older inequality, and the estimate associated with the continuous embedding $\dot{B}^2_{2,1}(A_D) \hookrightarrow \dot{B}^1_{\infty,1}(A_D)$, we obtain
\[
\begin{split}
& 
	\Big| \int _{\Omega} \Big( \Delta _D N_0 \big(  (1-S_j) \theta,  \theta _\mu \big) \Big)
\psi _j (\Lambda _D)\Delta  \theta _\mu ~dx 
\Big| 
\\
\leq 
&C2^j \sum_{k >j} 2^k \| \phi_k(\Lambda _D)\theta \|_{L^2}   
    \| \theta_\mu \|_{\dot B^2_{2,1}}
    \cdot \| \psi_j (\Lambda _D)^{\frac{1}{2}}  \Delta \theta_\mu \|_{L^2}. 
\end{split}
\]
Summing over $j \in \mathbb{Z}$ after dividing by $\| \psi_j (\Lambda_D)^{1/2} \Delta \theta \|_{L^2}$, we conclude that \eqref{1126-1} holds.
\end{pf}

\begin{prop}\label{prop:1010-5}
		Let $N_\mu$ be defined by \eqref{1126-3} for $\mu > 0$. 
We consider the following integral equation:
	\begin{equation}\label{1010-2}
		\theta(t) = \theta_0 - \int_0^t N_\mu (\theta (\tau) , \theta (\tau))  ~ d\tau.
	\end{equation}
	For every $\theta_0 \in \dot{B}^2_{2,1}(A_D)$, there exists a time $T > 0$ such that the integral equation \eqref{1010-2} has a unique solution $\theta^{(\mu)} \in C([0,T], \dot{B}^2_{2,1}(A_D))$. 
	Furthermore, the existence time $T$ depends only on $\| \theta_0 \|_{\dot{B}^2_{2,1}(A_D)}$ and is independent of $\mu > 0$. In particular, the solutions $\{ \theta^{(\mu)} \}_{\mu > 0}$ satisfy the uniform estimate
	\begin{equation}\label{0219-2}
	\sup_{\mu > 0} \sum_{j \in \mathbb{Z}} 2^{2j} \| \phi_j(\Lambda_D) \theta^{(\mu)} \|_{L^\infty(0,T; L^2)} < \infty.
\end{equation}
\end{prop}

\begin{pf}
\noindent 
Step 1.  (The existence of a solution with $T$ depending on $\mu$) 
Let $\theta _0 \in \dot B^2_{2,1}(A_D)$ and $T > 0$. For $\theta \in L^\infty(0, T; \dot{B}^2_{2,1}(A_D))$, we define  $\Psi = \Psi(\theta)$ by
\[
\Psi(\theta)(t) = \theta_0 - \int_0^t N_\mu(\theta, \theta)(\tau) \, d\tau,
\]
where $N_\mu(\theta, \tilde{\theta})$ is defined as in \eqref{1126-3}. We aim to find a fixed point of $\Psi$ by applying the Banach fixed point theorem in the space $X_T := L^\infty (0,T ; \dot B^2_{2,1}(A_D))$ equipped with the norm
\[
\| \theta \|_{X_T} := \sup_{t \in [0,T]} \| \theta(t) \|_{\dot{B}^2_{2,1}(A_D)}.
\]

We claim the following bilinear estimate:
\begin{equation}\label{0216-1}
\| N_\mu (\theta , \widetilde \theta ) \|_{\dot B^2_{2,1}(A_D)} 
\leq C \mu ^{-1+\frac{\gamma}{2}} \| \theta  \|_{\dot B^2_{2,1}(A_D)} \| \widetilde{\theta} \|_{\dot B^2_{2,1})(A_D)},
\end{equation}
where $C$ is independent of $\mu > 0$. Let $0 < \gamma < 1/2$. We note that
\[
\begin{split}
	&\| (1+\mu A_D)^{-1} \|_{L^2 \to L^2} \leq 1 , 
\\
& 
\| A_D ^{1- \frac{\gamma}{2}} (1 + \mu A_D)^{-1} \|_{L^2 \to L^2} 
\leq C \mu ^{-1 + \frac{\gamma}{2}}, 
\end{split}
\]
which, combined with Proposition~\ref{prop:1121-1}, imply that
\[
\begin{split}
	\| N_\mu (\theta , \widetilde\theta) \|_{\dot B^2_{2,1}(A_D)}
\leq 
& C \mu ^{-1+\frac{\gamma}{2}}
\Big\| \Lambda _D^{\gamma} 
 \Big(  (\nabla ^\perp \Lambda _D^{-1} \theta \cdot \nabla ) (1+\mu A_D)^{-1} \widetilde \theta 
 \Big)  \Big\|_{\dot B^0_{2,1}(A_D)}
\\
\leq 
& C \mu ^{-1+\frac{\gamma}{2}}
\big( \| \theta \|_{\dot B^\gamma_{2,1}(A_D)} \| \widetilde{\theta} \|_{\dot B^1_{\infty,1}(A_D)} 
+ \| \theta \|_{\dot B^0_{\infty,1}(A_D)} \| \widetilde \theta \|_{\dot B^{1+\gamma}_{2,1}(A_D)}
\big) .
\end{split}
\]
Applying the embeddings $\dot{B}^2_{2,1}(A_D) \hookrightarrow \dot{B}^{s+\gamma}_{2,1}(A_D)$ and $\dot{B}^2_{2,1}(A_D) \hookrightarrow \dot{B}^s_{\infty,1}(A_D)$ for $s = 0, 1$, we obtain \eqref{0216-1}.

It follows from \eqref{0216-1} that 
\[
\| \Psi (\theta) \|_{X_T} 
\leq \| \theta_0 \|_{\dot B^2_{2,1}(A_D)} 
+ \Big\| \int_0^t N_\mu (\theta, \theta)  d\tau \Big\|_{X_T}  
\leq \| \theta_0 \|_{\dot B^2_{2,1}}   
+ C \mu ^{-1+\frac{\gamma}{2}}T \| \theta \|_{X_T} ^2. 
\]
Similarly, for the difference between two terms, we have
\[
\Big\| \int_0^t N_\mu (\theta, \theta)  d\tau  - \int_0^t N_\mu (\widetilde \theta, \widetilde \theta)  d\tau \Big\|_{X_T} 
\leq C \mu ^{-1+\frac{\gamma}{2}}T 
 (\| \theta \|_{X_T} + \| \widetilde \theta \|_{X_T}) 
   \| \theta - \widetilde \theta \|_{X_T}. 
\]
By a standard fixed point argument, there exists a unique local solution $\theta = \theta^{(\mu)}$ in $L^\infty (0, T; \dot{B}^2_{2,1}(A_D))$, where the existence time $T$ depends on $\mu > 0$.

We also obtain the time continuity, i.e., $\theta^{(\mu)} \in C([0,T]; \dot B^2_{2,1}(A_D))$, 
from the integrability that 
\[
\int _0^T \| N_\mu (\theta , \theta ) \|_{\dot B^2_{2,1}(A_D)} ~dt < \infty 
\quad \text{for } \theta \in L^\infty (0,T ; \dot B^2_{2,1}(A_D)) . 
\]

\vskip3mm 

\noindent Step 2.  
Let $\theta ^{(\mu)}$ be the solution constructed in Step 1. 
To establish a uniform existence time, we will show that
\begin{equation}\label{1126-6}
\| \theta^{(\mu)} (t) \|_{\dot B^2_{2,1}(A_D)} \leq C  \| \theta _ 0 \|_{\dot B^2_{2,1}(A_D)} 
+ C T \| \theta^{(\mu)}  \|_{X_T}^2 ,
\end{equation}
where $C$ is a positive constant independent of $\mu > 0$. Once \eqref{1126-6} is established, a standard argument ensures the existence of a time $T > 0$ satisfying
\[
T \geq \dfrac{C}{\| \theta _0 \|_{\dot B^2_{2,1}(A_D)}},
\]
where $C$ is independent of $\mu$. Therefore, it suffices to prove \eqref{1126-6}.

Since $N_\mu (\theta , \theta)$ contains the resolvent $(1+\mu A_D)^{-1}$, it is easy to see that $\partial _t \Delta \theta (t)$ makes sense and 
belongs to $L^\infty (0,T ; L^2)$ at least. 
Applying $\Delta$ to both sides of \eqref{1010-2}, we have
\[
\partial _t \Delta \theta ^{(\mu)} + \Delta N_\mu (\theta^{(\mu)} , \theta ^{(\mu)}) = 0. 
\]
For each $j \in \mathbb{Z}$, multiplying the equation by $\psi_j (\Lambda_D) \Delta \theta^{(\mu)}$, where $\psi_j(\Lambda_D)$ is the operator defined via the resolvent, and integrating over the domain $\Omega$, we obtain
\[
\begin{split}
\dfrac{1}{2} \partial _t \|  \psi _j (\Lambda _D)^{\frac{1}{2}} \Delta \theta ^{(\mu)} \|_{L^2} ^2 
	\leq  
	\Big|\int _{\Omega} \Big( \Delta  N_\mu (\theta^{(\mu)} , \theta^{(\mu)}) \Big) \psi _j (\Lambda _D) \Delta \theta^{(\mu)} ~dx \Big| .
\end{split}
\]
Dividing both sides by $\| \psi _j (\Lambda _D)^{1/2} \Delta \theta ^{(\mu)}\|_{L^2}$ and integrating with respect to the time variable, we obtain
\begin{equation}\label{0219-1}
\begin{split}
&	\dfrac{1}{2} \|  \psi _j (\Lambda _D)^{\frac{1}{2}} \Delta \theta ^{(\mu)} (t)\|_{L^2} 
\\
	\leq  
&	\dfrac{1}{2} \|  \psi _j (\Lambda _D)^{\frac{1}{2}} \Delta \theta_0\|_{L^2} 
\\
&	+ \int_0 ^t 
	  \Big|  \int _{\Omega} \Big( \Delta  N_\mu (\theta^{(\mu)} , \theta^{(\mu)}) \Big) \psi _j (\Lambda _D) \Delta \theta^{(\mu)}~dx\Big| 
	\cdot \dfrac{1}{\| \psi _j (\Lambda _D)^{\frac{1}{2}} \Delta \theta ^{(\mu)} \|_{L^2}} 
	~d\tau .
\end{split}
\end{equation}
By writing $N_\mu (\theta^{(\mu)} , \theta^{(\mu)}) = N_\mu (S_j \theta^{(\mu)} , \theta^{(\mu)}) + N_\mu ( (1-S_j) \theta^{(\mu)} , \theta^{(\mu)})$, taking the sum over $j \in \mathbb Z$, and applying Proposition~\ref{lem:1126-5}, we see that
\[
\sum _{j \in \mathbb Z}\| \psi _j (\Lambda _D) ^{\frac{1}{2}} \Delta \theta^{(\mu)} (t) \|_{L^2} 
\leq \sum _{j \in \mathbb Z} \| \phi_j(\Lambda _D)^{\frac{1}{2}}\Delta \theta _0 \|_{L^2}
+ C \int_0^t \| \theta^{(\mu)} (\tau) \|_{\dot B^2_{2,1}(A_D)} ^2 ~d\tau . 
\]
We conclude \eqref{1126-6} because the norm defined by $\psi_j(\Lambda _D)^{1/2} \Delta$ is equivalent to the original one by Lemma~\ref{lem:0128-3} (2).

\vskip3mm 

\noindent 
Step 3. By taking the supremum of both sides of \eqref{0219-1} with respect to $t \in [0,T]$ before summing over $j \in \mathbb Z$, we obtain 
\[
\begin{split}
\sum _{j \in \mathbb Z}\| \psi _j (\Lambda _D) ^{\frac{1}{2}} \Delta \theta ^{(\mu)} \|_{L^\infty (0,T ; L^2)} 
\leq 
& \sum _{j \in \mathbb Z} \| \phi_j(\Lambda _D)^{\frac{1}{2}}\Delta \theta _0 \|_{L^2}
+ C \int_0^T \| \theta^{(\mu)} (\tau) \|_{\dot B^2_{2,1}(A_D)} ^2 ~d\tau . 
\\
\leq 
& C \| \theta _0 \|_{\dot B^2_{2,1}(A_D)} 
  + C \| \theta ^{(\mu)} \|_{L^\infty (0,T; \dot B^2_{2,1}(A_D))}^2 < \infty ,
\end{split}
\]
which proves \eqref{0219-2}. 
\end{pf}

\vskip3mm 

\noindent 
{\bf Proof of Theorem~\ref{thm:1}. } 
Let $\theta_0 \in \dot B^2_{2,1}(A_D)$ and $\mu > 0$. It follows from Proposition~\ref{prop:1010-5} that there exists a solution $\theta ^{(\mu)} \in C([0,T]; \dot B^2_{2,1}(A_D))$ to the equation $\partial _t \theta ^{(\mu)} + N_\mu (\theta^{(\mu)}, \theta^{(\mu)}) = 0$ with the initial condition $\theta ^{(\mu)}(0) = \theta _0$. Here, the time $T > 0$ is independent of $\mu$, and the solution satisfies the uniform bound \eqref{0219-2}.

We establish the existence of a subsequence of $\{ \theta ^{(\mu)} \}_{\mu > 0}$ that converges in the weak-$*$ topology of $\widetilde L^\infty (0,T ; \dot B^2_{2,1}(A_D))$, where 
\[
\widetilde L^\infty (0,T ; \dot B^2_{2,1}(A_D)) 
:= 
\Big\{
\theta \in L^\infty (0,T ; \dot B^2_{2,1}(A_D)) \, \Big| \,  
\sum _{j \in \mathbb Z} 2^{2j} \| \phi_j(\Lambda _D) \theta \|_{L^\infty (0,T ; L^2)}
< \infty
\Big\}. 
\]
Let $c_0 (\mathbb Z ; L^1 (0,T ; L^2))$ be the space of sequences $\{f_j\}_{j \in \mathbb Z}$ of measurable functions on $(0,T ) \times \Omega$ satisfying 
\[
\sup _{ j \in \mathbb Z} \| f_j \|_{L^1 (0,T ; L^2)} < \infty 
\quad \text{and} \quad 
\lim_{|j| \to \infty}  \| f_j \|_{L^1(0,T; L^2)} = 0 . 
\]
We consider the mapping 
\[
\widetilde L^\infty (0,T ; \dot B^2_{2,1}(A_D)) 
\ni \theta \quad \mapsto \quad 
\{2^{2j} \phi_j (\Lambda _D)\theta \}_{j \in \mathbb Z} 
\in \ell ^1 (\mathbb Z ; L^\infty (0,T ; L^2)).
\]
This map is an isometry, and its range is a closed subspace of $\ell^1 (\mathbb Z ; L^\infty (0,T ; L^2))$. 
We note that $\widetilde L^\infty (0,T ; \dot B^2_{2,1}(A_D))$ can be identified with the dual of a closed subspace of $c_0 (\mathbb Z ; L^1 (0,T ; L^2))$. Since this predual is separable, the Banach-Alaoglu theorem ensures that for any bounded sequence in $\widetilde L^\infty (0,T; \dot B^2_{2,1}(A_D))$, there exists a subsequence that converges in the weak-$*$ topology.

We can then find a sequence $\{\mu _N\}_{ N \in \mathbb N}$ of positive numbers with $\mu _N \to 0 $ as $N \to \infty$ and a limit function $\theta$ such that $\theta^{(\mu_N)}$ converges to $\theta$ in the weak-$*$ topology of $\widetilde L^\infty (0,T ; \dot B^2_{2,1}(A_D))$. We also note that 
\begin{equation}\label{0220-1}
	\sup _{0 < \mu \leq 1} 
	\Big( \| \partial _t \theta^{(\mu)}  \|_{L^\infty (0,T ; L^2)} 
	+ \|  \nabla ^\alpha \theta ^{(\mu)}\|_{L^\infty (0,T; L^2)} 
	\Big) 
	< \infty , \quad \alpha = 0, 1, 2, 
\end{equation}
which, combined with compactness arguments in Sobolev spaces on the bounded domain $(0,T) \times \Omega$, ensures the strong convergence of a subsequence of $\{ \theta ^{(\mu_N)}\}_{N \in \mathbb N}$ in $L^\infty (0,T ; L^2)$. 

The bound \eqref{0220-1} is proved as follows. The estimate for the time derivative is obtained directly from the equation. Indeed, by applying Lemma~\ref{lem:0128-3} (1) and Proposition~\ref{prop:0128-2} (4), we have 
\[
\begin{split}
	\| \partial _t \theta ^{(\mu)}  \|_{L^2}
	&= \| N_\mu (\theta ^{(\mu)} , \theta ^{(\mu)} ) \|_{L^2}
	\leq \| \nabla ^\perp \Lambda _D^{-1} \theta ^{(\mu)} \|_{L^\infty} 
	\| \nabla \theta ^{(\mu)} \|_{L^2}
	\\
	&\leq 
	C \| \theta ^{(\mu)} \|_{\dot B^1_{2,1}(A_D)} ^2 
	\leq C \| \theta ^{(\mu)}  \|_{\dot B^2_{2,1}(A_D)} ^2 ,
\end{split}
\]
where the right-hand side is uniformly bounded with respect to $\mu > 0$ by Proposition~\ref{prop:1010-5}. The spatial derivatives are also bounded as 
\[
\| \nabla^\alpha  \theta^{(\mu)} \|_{L^2} 
\leq C \| \theta ^{(\mu)}  \|_{\dot B^\alpha _{2,1}(A_D)} 
\leq C \| \theta ^{(\mu)}  \|_{\dot B^2_{2,1}(A_D)}, 
\quad \alpha = 0, 1, 2. 
\]
Therefore, by passing to a further subsequence if necessary, we may assume that $\{ \theta ^{(\mu_N)} \}$ and $\{ \nabla \theta ^{(\mu_N)} \}$ converge to $\theta$ and $\nabla \theta$, respectively, in $L^\infty (0,T ; L^2)$ as $N \to \infty$.

Next, we take the limit in the integral equation
\[
\theta ^{(\mu)}  = \theta _ 0 - 
\int _0^t N_\mu (\theta ^{(\mu)}, \theta ^{(\mu)}) ~d\tau  
\quad \text{in } L ^\infty (0,T ; L^2). 
\]
The convergence established in the previous argument ensures that the limit function $\theta$ satisfies 
\begin{gather}\notag 
	\theta  = \theta _ 0 +
	\int _0^t N_0 (\theta , \theta) ~d\tau  
	\quad \text{in } L ^\infty (0,T ; L^2), 
	\\ \label{0220-5}
	\sum _{j \in \mathbb Z} 2^{2j} \| \phi_j(\Lambda _D) \theta  \|_{L^\infty (0,T ; L^2)} 
	< \infty .
\end{gather}
Since $\theta$ satisfies the above integral equation, it follows that $\theta \in C([0,T ], L^2)$. Thus, we obtain a solution $\theta \in L^\infty (0,T ; \dot B^2_{2,1}(A_D)) \cap C([0,T], L^2)$.

We now prove the time continuity in $C([0,T], \dot B^2_{2,1}(A_D))$. Let $\varepsilon > 0$ and $t_0 \in [0,T]$. To show the continuity at $t = t_0$, we first note from \eqref{0220-5} that there exists $j_0 \in \mathbb N$ such that
\[
\sum _{|j| > j_0} 2^{2j} \| \phi_j(\Lambda _D) \theta  \|_{L^\infty (0,T; L^2)} < \varepsilon .
\] 
From the continuity $\theta \in C([0,T], L^2)$, there exists $\delta > 0$ such that $|t-t_0| < \delta$ implies 
\[
2^{2j_0}\| \theta (t) - \theta (t_0) \|_{L^2} < \varepsilon.
\]
Therefore, for any $t$ satisfying $|t-t_0| < \delta$, we have
\[
\begin{split}
	\| \theta (t) - \theta(t_0)  \|_{\dot B^2_{2,1}(A_D)}
	&= \Big( \sum _{|j| > j_0} + \sum _{|j| \leq j_0} \Big) 2^{2j} \| \phi_j(\Lambda _D) (\theta(t) - \theta (t_0)) \|_{L^2} \\
	&\leq 2\varepsilon + \sum _{|j| \leq j_0} 2^{2j}  \cdot C \| \theta (t) - \theta (t_0) \|_{L^2} \\
	&\leq 2\varepsilon + C 2^{2 j_0} \| \theta (t) - \theta (t_0) \|_{L^2} < 2 \varepsilon + C \varepsilon,
\end{split}
\]
which implies the time continuity in $C([0,T], \dot B^2_{2,1}(A_D))$. 
We also observe that $\partial_t \theta \in C([0,T]; \dot B^1_{2,1}(A_D))$. 
In fact, by applying the bilinear estimate
\[
\| (\nabla^{\perp} \Lambda_D^{-1} f \cdot \nabla) g \|_{\dot B^1_{2,1}(A_D)}
\leq C \| f \|_{\dot B^2_{2,1}(A_D)} \| g \|_{\dot B^2_{2,1}(A_D)},
\]
together with the relation $\partial_t \theta = -(\nabla^{\perp} \Lambda_D^{-1} \theta \cdot \nabla) \theta$, we can ensure that the solution $\theta$ belongs to $C^1([0,T]; \dot B^1_{2,1}(A_D))$.

Finally, we prove the uniqueness of the solution. Let $\theta$ satisfy 
\[
\theta \in C([0,T]; W^{1,\infty} (\Omega)) \cap C^1 ([0,T]; L^2), 
\quad 
\partial _t \theta  + (\nabla ^\perp \Lambda _D^{-1} \theta \cdot \nabla) \theta = 0 , 
\quad \theta (0) = \theta _0.
\]
Suppose that $\widetilde{\theta} \in C([0,T]; W^{1,\infty} (\Omega)) \cap C^1 ([0,T]; L^2)$ is another solution with the same initial data. Then the difference $w := \theta - \widetilde{\theta}$ satisfies
\[
\partial _t w + (\nabla ^\perp \Lambda _D^{-1}w \cdot \nabla ) \theta 
+ (\nabla ^\perp \Lambda _D^{-1} \widetilde{\theta} \cdot \nabla) w 
= 0.
\]
Taking the $L^2$ inner product with $w$, we have
\[
\frac{1}{2} \frac{d}{dt} \| w \|_{L^2} ^2 
= - \int _{\Omega} \Big( \big( \nabla ^\perp \Lambda _D^{-1}w \cdot \nabla  \big) \theta \Big) \cdot w ~dx
\leq C \| \nabla \theta \| _{L^\infty} \| w \|_{L^2}^2, 
\]
where we used the fact that 
\[ 
\int_{\Omega} ((\nabla^\perp \Lambda_D^{-1} \widetilde{\theta} \cdot \nabla) w) \cdot w ~dx = 0. 
\] 
In the proof of the above property, we also note  that $\Lambda _D ^{-1} \widetilde \theta \in H^1_0 (\Omega)$  and that an approximation by smooth functions with compact support plays an important role to justify integration by parts. 
By Gronwall's inequality, we conclude that $w = 0$ on $[0,T]$, which proves the uniqueness.
\hfill $\qed$

\vskip10mm

%\noindent
%{\bf Acknowledgements. }
%The author was supported by the Grant-in-Aid for Young Scientists (A) (No.~17H04824)
%from JSPS. 
%

\vskip3mm 
%
%\noindent 
%{\bf Compliance with Ethical Standards}
%\begin{enumerate}
%\item
%\noindent 
%{\bf Conflict of Interest. }
%The author declares that he has no conflict of interest. 
%\item 
%Research involving human participants and/ or animals:
%
%This paper does not contain any studies with human participants or animals performed by any of the authors. 
%\item 
%Informed consent: 
%
%Informed consent was obtained from all individual participants included in the study.
%\end{enumerate}

\end{document}